\documentclass[10pt,reqno]{amsart}

\usepackage[T1]{fontenc}
\usepackage[utf8]{inputenc}
\usepackage{lmodern}

\usepackage{microtype}

\usepackage{comment}

\usepackage{amsmath}
\usepackage{amssymb}
\usepackage{amsthm}
\usepackage{mathtools}

\usepackage{bm}

\usepackage{mathrsfs}

\usepackage{enumitem}

\usepackage{graphicx}

\usepackage{booktabs}

\usepackage{xcolor}

\usepackage[colorinlistoftodos]{todonotes}

\usepackage[numbers]{natbib}
\usepackage[colorlinks=true, linkcolor=blue, citecolor=blue, urlcolor=blue]{hyperref}

\usepackage[nameinlink,capitalise,noabbrev]{cleveref}

\usepackage[a4paper,top=2.5cm,bottom=2.5cm,left=2.5cm,right=2.5cm,marginparwidth=2.5cm]{geometry}

\allowdisplaybreaks

\numberwithin{equation}{section}

\theoremstyle{plain}

\newtheorem{theorem}{Theorem}[section]

\newtheorem{lemma}[theorem]{Lemma}

\theoremstyle{definition}

\theoremstyle{remark}

\newtheorem{remark}[theorem]{Remark}

\crefname{theorem}{theorem}{theorems}
\Crefname{theorem}{Theorem}{Theorems}

\crefname{proposition}{proposition}{propositions}
\Crefname{proposition}{Proposition}{Propositions}

\crefname{lemma}{lemma}{lemmas}
\Crefname{lemma}{Lemma}{Lemmas}

\crefname{corollary}{corollary}{corollaries}
\Crefname{corollary}{Corollary}{Corollaries}

\crefname{definition}{definition}{definitions}
\Crefname{definition}{Definition}{Definitions}

\crefname{assumption}{assumption}{assumptions}
\Crefname{assumption}{Assumption}{Assumptions}

\crefname{condition}{condition}{conditions}
\Crefname{condition}{Condition}{Conditions}

\crefname{remark}{remark}{remarks}
\Crefname{remark}{Remark}{Remarks}

\crefname{example}{example}{examples}
\Crefname{example}{Example}{Examples}

\crefname{question}{question}{questions}
\Crefname{question}{Question}{Questions}

\crefname{equation}{equation}{equations}
\Crefname{equation}{Equation}{Equations}

\crefname{section}{section}{sections}
\Crefname{section}{Section}{Sections}

\newcommand{\R}{\mathbb{R}}
\newcommand{\N}{\mathbb{N}}

\newcommand{\inner}[2]{\langle #1, #2\rangle}
\newcommand{\biginner}[2]{\bigg\langle #1, #2\bigg\rangle}

\DeclareMathOperator{\tr}{tr}
\DeclareMathOperator{\Hess}{Hess}

\title[Grassmannian Brascamp--Lieb inequalities]
{From the Sphere to the Grassmannian:
Continuous Brascamp--Lieb Inequalities
via Measure Approximation and Heat Flow}

\author[M. Li]{Minguo Li}

\address{
Graduate School of Mathematical Sciences,
The University of Tokyo,
Tokyo, Japan
}

\email{minguoli@g.ecc.u-tokyo.ac.jp}

\date{}

\begin{document}

\begin{abstract}
We study continuous Brascamp--Lieb inequalities associated with isotropic measures on the sphere $S^{n-1}$ and on the Grassmannian $G(n,k)$.
We give two proofs of the continuous Grassmannian inequality.
The first proof approximates the continuous isotropic measure by discrete isotropic measures and then applies the equal-rank geometric Brascamp--Lieb inequality.
In the rank-one case, the construction of the discrete isotropic measures follows from Barthe's work, which is standard, whereas in the higher-rank case the construction of finite approximate isotropic measures is given by a Tchakaloff-type theorem.
Weak convergence, geometric Brascamp--Lieb inequality and Fatou's lemma then yield the continuous inequalities.
The second proof generalizes the heat-semigroup preservation method of Barthe and Huet to isotropic measures on $G(n,k)$.
We prove a heat-semigroup preservation theorem in the Grassmannian setting and show that this theorem yields a Grassmannian Brascamp--Lieb inequality.
At the end, we also give a sketch for the mixed-rank case of the continuous Brascamp--Lieb inequality.
\end{abstract}



\maketitle

\section{Introduction}
\label{sec:sec1-introduction}

The \textit{Brascamp--Lieb inequality} provides a common framework for several classical integral inequalities, including H\"older's inequality, the Loomis--Whitney inequality, and Young's convolution inequality.
For finite-dimensional Euclidean spaces $E,E_1,\dots,E_m$, surjective linear maps $B_i:E\to E_i$, exponents $c_i\ge 0$, and nonnegative measurable functions $f_i:E_i\to[0,\infty)$, it takes the form
	\begin{equation}
         \int_E\prod_{i=1}^m f_i(B_i x)^{c_i}\,dx
         \le
         \mathbf{BL}\prod_{i=1}^m\left(\int_{E_i}f_i(y)\,dy\right)^{c_i}.
         \label{eq:sec1-general-bl}
	\end{equation}
The best constant $\mathbf{BL}$ is associated with the datum $(E,(E_i)_i,(B_i)_i,(c_i)_i)$.
This constant is usually called the \textit{Brascamp--Lieb constant} and this datum is usually called the \textit{Brascamp--Lieb datum}.
Brascamp and Lieb and also Beckner introduced this general class of inequalities, Lieb established the fundamental role of Gaussian inputs, and Bennett, Carbery, Christ, and Tao developed its finiteness and structural theory; see \cite{Beckner1975,BCCT,BrascampLieb1976,Lieb1990}.
Next we introduce a special case, which plays a very important role in the structural theory.

\subsection{Geometric and continuous Brascamp--Lieb inequalities}
\label{subsec:sec1-geometric-continuous-bl}

The geometric Brascamp--Lieb inequality is the special case in which the maps are partial isometries and their associated orthogonal projections decompose the identity.
The rank-one case of this inequality was first proved by Keith Ball~\cite{Ball1989} to solve problems in convex geometry.
For example, the volume ratio problem of convex bodies.
This inequality was later generalized to mixed rank case by Barthe~\cite{Barthe1998}.

Let $E=\R^n$ and $E_i = \R^{n_i}$ for any positive integer $n$ and $n_i \le n$. 
If
	\begin{equation}
        B_iB_i^*=I_{n_i}
        ~~
        \text{for all }
        1\le i \le m,
        \quad
        \text{and}
        \quad
        \sum_{i=1}^m c_iB_i^*B_i=I_n,
        \label{eq:sec1-geometric-condition}
	\end{equation}
then the Brascamp--Lieb constant in \eqref{eq:sec1-general-bl} is equal to $1$; see \cite{Ball1989,Barthe1998,BCCT,CLL2004}.
For equal-rank data, let $V_i\in G(n,k)$ (here it is easier to understand if we say that each $V_i$ is a $k$-dimensional linear subspace of $\R^n$, we say $V_i\in G(n,k)$ to unify the language) and define $B_i:\R^n\to V_i$ by $B_ix=P_{V_i}x$, where the projection $P_{V_i}$ is regarded as an operator on $\R^n$.
Then $B_iB_i^*=I_{V_i}$ and $B_i^*B_i=P_{V_i}$.
The geometric condition then becomes
	\begin{equation}
         \sum_{i=1}^m c_iP_{V_i}=I_n,
         \qquad c_i>0,
         \label{eq:sec1-equal-rank-geometric-condition}
	\end{equation}
and the corresponding inequality is
	\begin{equation}
         \int_{\R^n}\prod_{i=1}^m f_i(P_{V_i}x)^{c_i}\,dx
         \le
         \prod_{i=1}^m\left(\int_{V_i}f_i(y)\,d\lambda_{V_i}(y)\right)^{c_i}.
         \label{eq:sec1-equal-rank-geometric-bl}
	\end{equation}
Here $\lambda_{V_i}$ denotes the $k$-dimensional Lebesgue measure on $V_i$ induced by the Euclidean structure of $\R^n$.
In the rank-one case, each $V_i$ is generated by a unit vector $u_i\in S^{n-1}$, and $P_{V_i}=u_i\otimes u_i$, where $(u\otimes u)x=\inner{x}{u}u$.
The corresponding identity decomposition to the rank-one geometric inequality is $\sum_i c_i u_i\otimes u_i=I_n$.
Equivalently, the discrete measure $\sum_i c_i\delta_{u_i}$ on $S^{n-1}$ satisfies an isotropy condition.
This observation leads naturally from finite geometric data to measures on the sphere.

A finite Borel measure $\mu$ on $S^{n-1}$ is called \textit{isotropic} if
	\begin{equation}
	    \int_{S^{n-1}}u\otimes u\,d\mu(u)=I_n.
	    \label{eq:sec1-spherical-isotropy}
	\end{equation}
Taking traces in \eqref{eq:sec1-spherical-isotropy} gives $\mu(S^{n-1})=n$.
Barthe replaced the discrete geometric datum by such an isotropic measure and obtained the continuous rank-one Brascamp--Lieb inequality \cite{Barthe2004}.
We use the conventions $\log 0=-\infty$ and $e^{-\infty}=0$ whenever inputs are allowed to vanish.

\begin{theorem}[Barthe \cite{Barthe2004}]
    \label{thm:sec1-forward-contin-rank-one-bl}
    Let $\mu$ be an isotropic finite Borel measure on $S^{n-1}$.
    Suppose that $F:S^{n-1}\times\R\to(0,\infty)$ is continuous and that
	\[
	    f_u(s)=\mathbf 1_{[\alpha(u),\beta(u)]}(s)F(u,s),
	\]
    where $\alpha(u)<\beta(u)$ for every $u\in S^{n-1}$, $\alpha$ is either real-valued and continuous or identically equal to $-\infty$, and $\beta$ is either real-valued and continuous or identically equal to $+\infty$.
    Assume further that there exists a nonnegative function $\mathcal F\in L^1(\R)\cap L^\infty(\R)$ such that $f_u(s)\le\mathcal F(s)$ for every $(u,s)\in S^{n-1}\times\R$.
    Then
	\begin{equation}
         \int_{\R^n}
             \exp\left(\int_{S^{n-1}}\log f_u\bigl(\inner{x}{u}\bigr)\,d\mu(u)\right)
         \,dx
         \le
         \exp\left(\int_{S^{n-1}}\log\left(\int_{\R}f_u(s)\,ds\right)\,d\mu(u)\right).
         \label{eq:sec1-forward-contin-rank-one-bl}
	\end{equation}
    The inner integral on the left-hand side is allowed to take the value $-\infty$.
\end{theorem}
In the same paper, Barthe also proved a reverse form of the preceding inequality.
The present paper considers only forward Brascamp--Lieb inequalities, so we do not state the reverse inequality here.
For further results on reverse Brascamp--Lieb inequalities, we refer the reader to~\cite{Barthe1998,Barthe2004,CourtadeLiu2021}.

Before proceeding, we briefly comment on the hypotheses imposed on the input functions in Theorem~\ref{thm:sec1-forward-contin-rank-one-bl}.
The joint continuity of $F$ and the continuity of the endpoint functions $\alpha$ and $\beta$ ensure the required measurability and continuity with respect to $u$. 
The common domination in $L^1(\R)\cap L^\infty(\R)$ provides the uniform bounds needed for the limiting arguments and, after applying heat-flow, for differentiation under the integral sign.
At the same time, this class of inputs is sufficiently broad for several applications in convex geometry.
Lutwak, Yang, and Zhang observed that their volume inequalities for isotropic measures can also be obtained from Barthe's continuous Brascamp--Lieb inequality~\cite{LutwakYangZhang2007}---a basic class of input functions arising from intersections of symmetric strips is obtained, for a fixed $t>0$, by taking
	\[
	    F(u,s)=1,
	    \qquad
	    \alpha(u)=-t,
	    \qquad
	    \beta(u)=t,
	    \qquad
	    \mathcal F(s)=\mathbf 1_{[-t,t]}(s).
	\]
A Gaussian version,
	\[
	    f_u(s)=\mathbf 1_{[-t,t]}(s)e^{-s^2/2},
	\]
was used by Sadovsky and Zhang to derive a Gaussian-measure comparison for origin-symmetric convex bodies in John's position~\cite{SadovskyZhang2025}.
Brazitikos and Liakopoulos applied a refinement of the continuous Brascamp--Lieb inequality with the full-support inputs
	\[
	    f_u(s)=\exp\left(-\frac{|s|^p}{d}\right)
	\]
to estimate volumes of $L_p$-type bodies arising in the study of visibility and transversality~\cite{BrazitikosLiakopoulos2025}.
All of the inputs above satisfy the hypotheses of Theorem~\ref{thm:sec1-forward-contin-rank-one-bl}.

Barthe's proof uses the continuous Ball--Barthe inequality \cite{Barthe2004}.
One purpose of the present paper is to recover Theorem~\ref{thm:sec1-forward-contin-rank-one-bl} directly from the rank-one geometric inequality by approximating $\mu$ with discrete isotropic measures.

\subsection{The continuous Grassmannian inequality}
\label{subsec:sec1-grassmannian-main-results}

The equal-rank geometric inequality suggests replacing the unit sphere by the Grassmannian $G(n,k)$ of $k$-dimensional linear subspaces of $\R^n$.
For $V\in G(n,k)$, let $P_V:\R^n\to\R^n$ be the orthogonal projection onto $V$.
A finite Borel measure $\mu$ on $G(n,k)$ is called \textit{isotropic} if
	\begin{equation}
	    \int_{G(n,k)}P_V\,d\mu(V)=I_n.
	    \label{eq:sec1-grassmannian-isotropy}
	\end{equation}
Since $\tr P_V=k$, taking traces in \eqref{eq:sec1-grassmannian-isotropy} yields $\mu(G(n,k))=n/k$.
The input family is naturally regarded as a function on the tautological bundle
	\[
	    T_{n,k}:=\{(V,y)\in G(n,k)\times\R^n:y\in V\}.
	\]
We give two proofs of a Grassmannian Brascamp--Lieb inequality.
The first uses measure approximation, while the second is based on heat-semigroup preservation.
Li, Xi, and Zhang established Grassmannian analogues of the Ball--Barthe inequality and observed that these inequalities could be used to derive a Grassmannian Brascamp--Lieb inequality; see~\cite{LiXiZhang2017}.
We do not pursue this approach here.

Our first main statement concerns strictly positive continuous inputs.

\begin{theorem}
    \label{thm:sec1-forward-grassmannian-bl}
    Let $1\le k\le n$, and let $\mu$ be an isotropic Borel measure on $G(n,k)$.
    Suppose that $F:T_{n,k}\to(0,\infty)$ is continuous and that there exists a function $\mathcal F:[0,\infty)\to[0,\infty)$ satisfying
	\begin{equation}
	    \mathcal F\in L^\infty([0,\infty)),
	    \qquad
	    \int_0^\infty\mathcal F(r)r^{k-1}\,dr<\infty,
	    \qquad
	    F(V,y)\le\mathcal F(|y|).
	    \label{eq:sec1-grassmannian-radial-majorant}
	\end{equation}
    Then
	\begin{equation}
	    \int_{\R^n}
	        \exp\left(\int_{G(n,k)}\log F(V,P_Vx)\,d\mu(V)\right)
	    \,dx
	    \le
	    \exp\left(
	        \int_{G(n,k)}
	            \log\left(\int_VF(V,y)\,d\lambda_V(y)\right)
	        \,d\mu(V)
	    \right).
	    \label{eq:sec1-forward-grassmannian-bl}
	\end{equation}
    Here, $\lambda_V$ denotes the $k$-dimensional Lebesgue measure on $V$ induced by the Euclidean structure of $\R^n$.
\end{theorem}

We also prove the following compactly supported form.
For $r>0$, the notation $B_V(0,r)$ refers throughout to the closed ball in $V$ centered at the origin with radius $r$.

\begin{theorem}
    \label{thm:sec1-compact-grassmannian-bl}
    Under the assumptions of Theorem~\ref{thm:sec1-forward-grassmannian-bl}, let $\alpha:G(n,k)\to(0,\infty)$ be continuous and define
	\begin{equation}
	    f_V(y):=\mathbf 1_{B_V(0,\alpha(V))}(y)F(V,y).
	    \label{eq:sec1-compact-grassmannian-input}
	\end{equation}
    Then
	\begin{equation}
	    \int_{\R^n}
	        \exp\left(\int_{G(n,k)}\log f_V(P_Vx)\,d\mu(V)\right)
	    \,dx
	    \le
	    \exp\left(
	        \int_{G(n,k)}
	            \log\left(\int_Vf_V(y)\,d\lambda_V(y)\right)
	        \,d\mu(V)
	    \right).
	    \label{eq:sec1-compact-grassmannian-bl}
	\end{equation}
\end{theorem}

\begin{remark}
The preceding theorem can also be deduced from the abstract framework of Cordero-Erausquin and Maurey \cite[Theorem~5.1 and the discussion following it]{CorderoErausquinMaurey2017}.
We explain the relevant choices at a formal level.
We restrict here to the non-truncated case $\alpha\equiv\infty$.

Let $m(CM)$ and $n(CM)$ denote the dimensions of the source and target Euclidean spaces in their theorem.
The relevant conclusion of their theorem is
\begin{equation}
	\int_{\Omega_2}
	-\log\left(
		\int_{\R^{n(CM)}}
		e^{-g_t(y)}\,d y
	\right)\,d\mu_2(t)
	\le
	\int_{\Omega_1}
	-\log\left(
		\int_{\R^{m(CM)}}
		e^{-f_s(x)}\,d x
	\right)\,d\mu_1(s).
	\label{eq:sec1-CorderoErausquin-Maurey-equalrank}
\end{equation}
We take the source parameter space to be a singleton and the target parameter space to be the Grassmannian:
\[
	X_1=(\{\ast\},\delta_\ast),
	\qquad
	X_2=(G(n,k),\mu).
\]
We also take
\[
	m(CM)=n,
	\qquad
	n(CM)=k.
\]
For each $V\in G(n,k)$, choose a linear isometry
\[
	U_V:\R^k\longrightarrow\R^n
\]
whose range is $V$.
The family $V\mapsto U_V$ may be chosen to be Borel measurable.
Define
\[
	Q_V:=U_V^*:\R^n\longrightarrow\R^k.
\]
Then
\[
	Q_VQ_V^*
	=
	U_V^*U_V
	=
	I_k,
\]
and
\[
	Q_V^*Q_V
	=
	U_VU_V^*
	=
	P_V.
\]
In particular,
\[
	U_VQ_V=P_V.
\]
Define
\[
	A:\R^n\longrightarrow L^2(G(n,k),\mu;\R^k)
\]
by
\[
	(Ax)(V):=Q_Vx.
\]
The isotropy condition gives
\[
	\begin{aligned}
	\|Ax\|_{L^2(G(n,k),\mu;\R^k)}^2
	&=
	\int_{G(n,k)}|Q_Vx|^2\,d\mu(V)\\
	&=
	\int_{G(n,k)}|P_Vx|^2\,d\mu(V)\\
	&=
	\left\langle
		\left(
			\int_{G(n,k)}P_V\,d\mu(V)
		\right)x,x
	\right\rangle\\
	&=
	|x|^2.
	\end{aligned}
\]
Thus, $A$ is an isometry.
Moreover, taking traces in the isotropy condition gives
\[
	k\mu(G(n,k))
	=
	\int_{G(n,k)}\operatorname{tr}(P_V)\,d\mu(V)
	=
	n.
\]
Hence, the dimension-balance condition in the theorem of Cordero-Erausquin and Maurey becomes
\[
	m(CM)\mu_1(\Omega_1)
	=
	n
	=
	k\mu(G(n,k))
	=
	n(CM)\mu_2(\Omega_2).
\]
For $V\in G(n,k)$ and $z\in\R^k$, define
\[
	g_V(z):=-\log F(V,U_Vz).
\]
Define the function on the source space by
\[
	\begin{aligned}
	f_\ast(x)
	&:=
	\int_{G(n,k)}g_V(Q_Vx)\,d\mu(V)\\
	&=
	-\int_{G(n,k)}
	\log F(V,U_VQ_Vx)\,d\mu(V)\\
	&=
	-\int_{G(n,k)}
	\log F(V,P_Vx)\,d\mu(V).
	\end{aligned}
\]
The functional hypothesis in their theorem then holds with equality:
\[
	\int_{G(n,k)}
	g_V((Ax)(V))\,d\mu(V)
	=
	f_\ast(x),
	\qquad
	x\in\R^n.
\]
Since $U_V:\R^k\to V$ is an isometry, the change of variables $y=U_Vz$ gives
\[
	\begin{aligned}
	\int_{\R^k}e^{-g_V(z)}\,d z
	&=
	\int_{\R^k}F(V,U_Vz)\,d z\\
	&=
	\int_VF(V,y)\,d\lambda_V(y).
	\end{aligned}
\]
On the source side, we have
\[
	\begin{aligned}
	\int_{\R^n}e^{-f_\ast(x)}\,d x
	&=
	\int_{\R^n}
	\exp\left(
		\int_{G(n,k)}
		\log F(V,P_Vx)\,d\mu(V)
	\right)\,d x.
	\end{aligned}
\]
Substituting these identities into \eqref{eq:sec1-CorderoErausquin-Maurey-equalrank} gives
\[
	\begin{aligned}
	&-\int_{G(n,k)}
	\log\left(
		\int_VF(V,y)\,d\lambda_V(y)
	\right)\,d\mu(V)\\
	&\qquad\le
	-\log\left[
		\int_{\R^n}
		\exp\left(
			\int_{G(n,k)}
			\log F(V,P_Vx)\,d\mu(V)
		\right)\,d x
	\right].
	\end{aligned}
\]
Multiplying both sides by $-1$ and then taking exponentials gives
\[
	\int_{\R^n}
	\exp\left(
		\int_{G(n,k)}
		\log F(V,P_Vx)\,d\mu(V)
	\right)\,d x
	\le
	\exp\left(
		\int_{G(n,k)}
		\log\left(
			\int_VF(V,y)\,d\lambda_V(y)
		\right)\,d\mu(V)
	\right).
\]
This is precisely the Grassmannian Brascamp--Lieb inequality in the preceding theorem.

Thus, the Grassmannian inequality is contained in the general framework of Cordero--Erausquin and Maurey. 
We nevertheless state it explicitly in Grassmannian form and give two direct proofs adapted to its geometric structure. 
The first proof is based on approximation by discrete measures, while the second is based on a direct heat-flow argument. 
Both are different from the stochastic variational argument used in \cite{CorderoErausquinMaurey2017}.
\end{remark}

\begin{remark}
The Grassmannian Brascamp--Lieb inequality~\eqref{eq:sec1-compact-grassmannian-bl} (also,~\eqref{eq:sec1-forward-grassmannian-bl}) should cover Barthe's forward continuous Brascamp--Lieb inequality~\eqref{eq:sec1-forward-contin-rank-one-bl} in a reasonable sense.
The Grassmannian $G(n,1)$ is naturally identified with the real projective space
    \[
        S^{n-1}/\{\pm\}.
    \]
Let
    \[
        \pi:S^{n-1}\to G(n,1),
        \qquad
        \pi(u):=\operatorname{span}\{u\},
    \]
be the quotient map. 
For the finite isotropic Borel measure $\mu$ on $G(n,1)$, define a new measure $\widetilde{\mu}$ on $S^{n-1}$ by
    \[
        \int_{S^{n-1}}
            \varphi(u)
        \,d\widetilde{\mu}(u)
        :=
        \int_{G(n,1)}
            \frac{1}{2}
            \sum_{u\in V\cap S^{n-1}}
                \varphi(u)
        \,d\mu(V)
    \]
for every bounded Borel function $\varphi$ on $S^{n-1}$. 
Since $V\cap S^{n-1}$ consists of two antipodal unit vectors, this definition is independent of any choice of orientation.
We can see $\widetilde{\mu}$ is also an isotropic measure on $S^{n-1}$.
Indeed, for $V=\operatorname{span}\{u\}$, one has
    \[
        P_V
        =
        u\otimes u
        =
        (-u)\otimes(-u),
    \]
and hence
    \[
        \int_{S^{n-1}}
            u\otimes u
        \,d\widetilde{\mu}(u)
        =
        \int_{G(n,1)}
            P_V
        \,d\mu(V)
        =
        I_n.
    \]
For $u\in S^{n-1}$ and $V=\pi(u)$, define
    \[
        g_u(s)
        :=
        f_V(su),
        \qquad
        s\in\R.
    \]
Then
    \(
        g_{-u}(s)
        =
        g_u(-s),
    \)
and since
    \(
        P_Vx
        =
        \inner{x}{u}u,
    \)
we have
    \(
        f_V(P_Vx)
        =
        g_u\bigl(\inner{x}{u}\bigr).
    \)
Moreover, the map $s\mapsto su$ is an isometry from $\R$ onto $V$, so
    \[
        \int_V
            f_V(y)
        \,d\lambda_V(y)
        =
        \int_{\R}
            g_u(s)
        \,ds.
    \]
Consequently, when $k=1$, Theorem~\ref{thm:sec1-compact-grassmannian-bl} is precisely the spherical rank-one inequality for the even measure $\widetilde{\mu}$ and the antipodally compatible family
    \[
        g_{-u}(s)
        =
        g_u(-s).
    \]
Conversely, every rank-one family satisfying this compatibility condition descends to a family of functions on the fibres of $T_{n,1}$. 
Thus, the case $k=1$ of the Grassmannian theorem is the natural unoriented rank-one version of Barthe's continuous inequality.
\end{remark}

Approximate forms of the geometric Brascamp--Lieb inequality have also been studied.
In the discrete rank-one setting, Brazitikos~\cite{Brazitikos2017} replaced the exact decomposition of the identity by an approximate one:
    \[
        I_n \preceq \sum c_i u_i\otimes u_i \preceq \gamma I_n
    \]
for $\gamma > 1$.
He obtained a Brascamp--Lieb inequality with modified weights and an additional coefficient depending on the approximation parameter.
Brazitikos and Giannopoulos later proved a continuous version for approximately isotropic Borel measures on the sphere~\cite{BrazitikosGiannopoulos2022}:
    \[
        I_n \preceq \int_{S^{n-1}} u\otimes u \,d\mu(u) \preceq \gamma I_n
    \]
for $\gamma > 1$.
Related perturbation results are known for finite Brascamp--Lieb data.
Bennett, Bez, Flock, and Lee~\cite{BennettBezFlockLee2018} proved that the Brascamp--Lieb constant is locally bounded under small perturbations of the underlying linear maps.
Bennett, Bez, Cowling, and Flock~\cite{BennettBezCowlingFlock2017} subsequently proved that, for fixed exponents, this constant is continuous with respect to the linear maps.
These results suggest that an approximate version of the Grassmannian inequality above may also hold.
More precisely, if
	\[
		A_\mu
		:=
		\int_{G(n,k)}P_V\,d\mu(V)
		\qquad\text{satisfies}\qquad
		\lVert A_\mu-I_n\rVert_{\mathrm{op}}
		\leq
		\varepsilon,
	\]
one may expect a corresponding inequality with suitably adjusted Brascamp--Lieb coefficients and an additional factor \(C_{n,k}(\varepsilon)\), where \(C_{n,k}(\varepsilon)\to1\) as \(\varepsilon\downarrow0\).
We leave this question for future work.

\subsection{Isotropy-preserving measure approximation}
\label{subsec:sec1-measure-approximation}
Barthe gave a detailed construction of discrete approximations to isotropic measures on the sphere in his work on the continuous Brascamp--Lieb inequality, see~\cite{Barthe2004}.
His purpose was to derive applications in convex geometry, and he did not pursue this approximation argument to obtain the continuous inequality.
Cordero-Erausquin and Maurey~\cite{CorderoErausquinMaurey2017} also commented that the continuous Brascamp--Lieb inequality can be derived by an approximation argument.
The aim of Section~\ref{sec:sec2-approximation-proofs} is to provide a complete implementation of this method and to extend it to the higher-rank setting, thereby proving a continuous Brascamp--Lieb inequality on the Grassmannian.

We first explain the approximation proof briefly.
Let $\mu$ be any isotropic Borel measure on $S^{n-1}$.
For any $\varepsilon > 0$, a Borel partition of $S^{n-1}$ into sets $U_i^\varepsilon$ of small diameter and a choice of points $u_i^\varepsilon\in U_i^\varepsilon$ give the discrete measure
	\begin{equation}
	 \mu_\varepsilon:=\sum_i\mu(U_i^\varepsilon)\delta_{u_i^\varepsilon}.
	 \label{eq:sec1-rank-one-uncorrected-measure}
	\end{equation}
This measure converges weakly to $\mu$, but its isotropy matrix
	\begin{equation}
	    Q_\varepsilon:=\int_{S^{n-1}}u\otimes u\,d\mu_\varepsilon(u)
	    \label{eq:sec1-rank-one-isotropy-matrix}
	\end{equation}
need not equal $I_n$.
Nevertheless, $Q_\varepsilon\to I_n$ in operator norm, so $Q_\varepsilon^{-1/2}\to I_n$ in operator norm.
Normalizing each vector after applying $Q_\varepsilon^{-1/2}$ and modifying its weight gives
	\begin{equation}
	    T_\varepsilon(u)
	    :=
	    \frac{Q_\varepsilon^{-1/2}u}{|Q_\varepsilon^{-1/2}u|},
	    \qquad
	    u\in S^{n-1}.
	    \label{eq:sec1-rank-one-correction-map}
	\end{equation}
	\begin{equation}
	    \nu_\varepsilon
	    :=
	    \sum_i
	        \mu(U_i^\varepsilon)|Q_\varepsilon^{-1/2}u_i^\varepsilon|^2
	        \delta_{T_\varepsilon(u_i^\varepsilon)}.
	    \label{eq:sec1-rank-one-corrected-measure}
	\end{equation}
The measure $\nu_\varepsilon$ is exactly isotropic and still converges weakly to $\mu$.
The geometric rank-one inequality, weak convergence, and Fatou's lemma then yield Theorem~\ref{thm:sec1-forward-contin-rank-one-bl}.

While we were considering the Grassmannian inequality, a paper by Sch\"afer and Ullrich appeared on arXiv.
Their main result, a generalized version of Tchakaloff's theorem~\cite[Theorem~1.1]{SchaeferUllrich2025}, provides a direct construction of discrete isotropic approximations.
Given a sufficiently fine Borel partition
	\[
	    G(n,k)=\bigsqcup_i U_i^\varepsilon,
	\]
we apply their theorem on each cell to the restriction $\mu|_{U^\varepsilon_i}$, and to the finite-dimensional family consisting of the constant function and the matrix entries of the map $V\mapsto P_V$.
This gives points $V_{i,\ell}^\varepsilon\in U_i^\varepsilon$ and nonnegative coefficients $c_{i,\ell}^\varepsilon$ such that
	\begin{align}
	    \sum_\ell c_{i,\ell}^\varepsilon
	    &=\mu(U_i^\varepsilon),
	    \label{eq:sec1-local-mass-preservation}\\
	    \sum_\ell c_{i,\ell}^\varepsilon P_{V_{i,\ell}^\varepsilon}
	    &=\int_{U_i^\varepsilon}P_V\,d\mu(V).
	    \label{eq:sec1-local-projection-preservation}
	\end{align}
After summing over the cells, the measure
	\[
	    \nu_\varepsilon:=\sum_{i,\ell}c_{i,\ell}^\varepsilon\delta_{V_{i,\ell}^\varepsilon}
	\]
is discrete and isotropic.
The compactness of $G(n,k)$ (continuous functions are uniformly continuous) and the finiteness of $\mu$ give the weak convergence $\nu_\varepsilon\rightharpoonup\mu$.

We then apply the equal-rank geometric inequality \eqref{eq:sec1-equal-rank-geometric-bl} to $\nu_\varepsilon$.
For each fixed $x\in\R^n$, weak convergence applies to $V\mapsto\log F(V,P_Vx)$.
On the right-hand side, it applies to
	\[
	    V\longmapsto\log\left(\int_VF(V,y)\,d\lambda_V(y)\right).
	\]
The continuity of this fiber integral is not immediate because the domain of integration varies with $V$.
It is proved by transporting the integral through isometries from a fixed subspace to nearby subspaces that converge to the identity.
Fatou's lemma handles the nonnegative integrands on the left-hand side, while weak convergence gives the limit on the right-hand side.
This proves Theorem~\ref{thm:sec1-forward-grassmannian-bl}.
Finally, we multiply the imput functions in Theorem~\ref{thm:sec1-forward-grassmannian-bl} by continuous cutoff functions and add them by small, smooth, integrable function.
This gives a family of approximating functions satisfying the hypotheses of Theorem~\ref{thm:sec1-forward-grassmannian-bl} and converging pointwise to the input family in Theorem~\ref{thm:sec1-compact-grassmannian-bl}.
Passing to the limit then yields Theorem~\ref{thm:sec1-compact-grassmannian-bl}.

\subsection{Heat-semigroup preservation}
\label{subsec:sec1-heat-semigroup-preservation}

Our second proof is based on heat flow.
Heat-flow proofs of geometric Brascamp--Lieb inequalities appear in the work of Carlen, Lieb, and Loss~\cite{CLL2004} and Bennett, Carbery, Christ, and Tao~\cite{BCCT}.
Barthe and Huet established a pointwise heat-semigroup preservation principle for geometric Brascamp--Lieb data \cite{BartheHuet2009}.
We generalize this preservation mechanism to an isotropic measure on $G(n,k)$.
An almost identical heat-semigroup argument also proves Barthe's continuous Brascamp--Lieb inequality, Theorem~\ref{thm:sec1-forward-contin-rank-one-bl}.
We do not include this proof in our paper.

Throughout the heat-flow argument, $e^{t\Delta}$ denotes the heat semigroup operator on $\R^n$.
For a function $f_V$ on $V$, the ambient heat flow of $f_V\circ P_V$ is constant in the directions orthogonal to $V$ and satisfies
	\begin{equation}
	    [e^{t\Delta}(f_V\circ P_V)](x)
	    =
	    (4\pi t)^{-k/2}
	    \int_V\exp\left(-\frac{|P_Vx-y|^2}{4t}\right)f_V(y)\,d\lambda_V(y).
	    \label{eq:sec1-fiber-heat-semigroup}
	\end{equation}
Thus it agrees with the heat flow of $f_V$ on the fiber $V$, evaluated at $P_Vx$.
The relevant preservation theorem is the following.

\begin{theorem}
    \label{thm:sec1-heat-semigroup-preservation}
    Let $\mu$ be an isotropic Borel measure on $G(n,k)$, and let $h:\R^n\to[0,\infty)$ be bounded, Lebesgue measurable, and compactly supported.
    Let $(f_V)_{V\in G(n,k)}$ be either the strictly positive family $f_V(y)=F(V,y)$ from Theorem~\ref{thm:sec1-forward-grassmannian-bl} or the compactly supported family from Theorem~\ref{thm:sec1-compact-grassmannian-bl}.
    If
	\begin{equation}
	    h(x)
	    \le
	    \exp\left(\int_{G(n,k)}\log f_V(P_Vx)\,d\mu(V)\right)
	    \qquad\text{for every }x\in\R^n,
	    \label{eq:sec1-heat-semigroup-initial}
	\end{equation}
    then, for every $x\in\R^n$ and $t>0$,
	\begin{equation}
	    (e^{t\Delta}h)(x)
	    \le
	    \exp\left(
	        \int_{G(n,k)}
	            \log\left([e^{t\Delta}(f_V\circ P_V)](x)\right)
	        \,d\mu(V)
	    \right).
	    \label{eq:sec1-heat-semigroup-positive-time}
	\end{equation}
\end{theorem}

We briefly describe the mechanism behind Theorem~\ref{thm:sec1-heat-semigroup-preservation}.
We first state the required maximum-principle lemma here; see \cite[Lemma~1]{BartheHuet2009} and \cite[Lemma~1]{Zhang2025}.

\begin{lemma}
    \label{lem:sec1-barthe-huet-zhang-maximum-principle}
    Let $E$ be a Euclidean space and let $C$ be continuous on $(0,\infty)\times E$, with $\partial_tC$, $\nabla_EC$, and $\Hess_EC$ continuous there.
    Assume that
	\begin{itemize}
	    \item[(i)] for every point $(t_*,x_*)\in(0,\infty)\times E$ such that
	\begin{equation}
	    \begin{cases}
	        C(t_*,x_*)\le 0,\\
	        \nabla_E C(t_*,x_*)=0,\\
	        \Hess_E C(t_*,x_*) \succeq 0
	    \end{cases}
	    \label{eq:sec1-minimum-point-conditions}
	\end{equation}
	    one has $\partial_t C(t_*,x_*)\ge 0$;
	    \item[(ii)] for every $T>0$,
	\begin{equation}
	    \liminf_{|x|_E\to \infty}
	        \left(
	            \inf_{t\in(0,T]}
	                C(t,x)
	        \right)
	        \ge
	        0;
	    \label{eq:sec1-spatial-boundary-condition}
	\end{equation}
	    \item[(iii)] for every $x_0\in E$,
	\begin{equation}
	    \liminf_{\substack{x\to x_0\\t\downarrow0}}
	        C(t,x)
	        \ge
	        0.
	    \label{eq:sec1-initial-boundary-condition}
	\end{equation}
	\end{itemize}
    Here all limit infima are allowed to take the value $\infty$.
    Then
	\[
	    C(t,x)
	    \ge
	    0,
	    \qquad
	    (t,x)\in(0,\infty)\times E.
	\]
\end{lemma}

For positive time, set
	\[
	    H(t,x):=\log[e^{t\Delta}h](x),
	    \qquad
	    F_V(t,x):=\log[e^{t\Delta}(f_V\circ P_V)](x),
	\]
and consider the comparison function
	\[
	    C(t,x):=\int_{G(n,k)}F_V(t,x)\,d\mu(V)-H(t,x).
	\]
    
Assumption~\eqref{eq:sec1-heat-semigroup-initial} is equivalent to the third condition of the preceding lemma and our goal~\eqref{eq:sec1-heat-semigroup-positive-time} is equivalent to the conclusion of the preceding lemma.
Therefore, the main task of Section~\ref{sec:sec3-heat-semigroup-proof} is to prove that the comparison function $C(t,x)$ verifies all hypotheses of that lemma.

The most interesting step in the proof is to verify that the comparison function \(C\) satisfies the first hypothesis of the preceding lemma.
Using the evolution identity---briefly speaking, if $e^U$ satisfies the heat equation, then
    \[
    	\partial_t U
    	=
    	\Delta U+|\nabla U|^2,
    \]
we obtain
    \[
    	\partial_tC
        =
        \left(
            \int_{G(n,k)}
                \Delta_{\R^n}F_V
            \,d\mu(V)
            -
            \Delta_{\R^n}H
        \right)
        +
        \left(
            \int_{G(n,k)}
                |\nabla_{\R^n}F_V|^2
            \,d\mu(V)
            -
            |\nabla_{\R^n}H|^2
        \right).
    \]
It is therefore enough to show that both the second-order and first-order derivative terms on the right-hand side are nonnegative.
This gives \(\partial_t C(t,x)\geq 0\) for any $t > 0$ and $x \in \R^n$.
The nonnegativity of the second-order term follows from the positive semidefiniteness of the Hessian.
For the first-order term, we use the condition \(\nabla C=0\), together with a Jensen-type inequality that will be proved later.
The complete argument, the verification of the second and third hypotheses, and the technical issues (for example, interchanging differentiation and integration) are given in Section~\ref{sec:sec3-heat-semigroup-proof}.

\subsection{Averaged Loomis--Whitney inequalities and quermassintegrals}

We explain an special case and an easy application of our Grassmannian inequality (although this application leads to a result weaker than the sharp estimate by Milman and Yehudayoff~\cite{MilmanYehudayoff2023}).
A special case of our continuous Grassmannian Brascamp--Lieb inequality is the logarithmic averaged Loomis--Whitney inequality appeared in \cite[Remark~8.3]{BennettTao2024} by Bennett and Tao.

Let $\nu_{n,k}$ denote the rotation-invariant probability measure on $G(n,k)$.
By rotation invariance, the operator
    \[
        \int_{G(n,k)}P_V\,d\nu_{n,k}(V)
    \]
is a scalar multiple of $I_n$.
Taking traces shows that
    \[
        \int_{G(n,k)}P_V\,d\nu_{n,k}(V)
        =
        \frac{k}{n}I_n.
    \]
Consequently, the measure
    \[
        d\mu(V):=\frac{n}{k}\,d\nu_{n,k}(V)
    \]
is isotropic.

Let $K\subset\R^n$ be a compact set with positive Lebesgue measure.
Choose $R>0$ such that $K\subset B_{\R^n}(0,R)$.
For every $m\ge1$, define
    \[
        F_m(V,y)
        :=
        \exp\left(-m\operatorname{dist}(y,P_VK)^2\right)
    \]
and
    \[
        f_{V,m}(y)
        :=
        \mathbf 1_{B_V(0,R)}(y)F_m(V,y).
    \]
We observe that $F_m > 0$ on $T_{n,k}$.
Since
    \[
        \operatorname{dist}(y,P_VK)
        =
        \min_{z\in K}|y-P_Vz|,
    \]
the function $F_m$ is continuous on $T_{n,k}$.
Indeed, to show the continuity, it suffices to show that $(V,y)\mapsto \min_{z\in K}|y - P_Vz|$ is continuous on $T_{n,k}$.
Suppose $(W,w)\to (V,y)$
    \begin{align*}
        \left|
            \min_{z\in K}
                |w - P_Wz|
            -
            \min_{z\in K}
                |y - P_Vy|
        \right|
        &\le
        \max_{z\in K}
            \big|
                |w - P_Wz| - |y - P_Vz|
            \big|\\
        &\le
        |w - y|
        +
        R\,\|P_V - P_W\|_{\operatorname{op}}
        \to
        0.
    \end{align*}    
Moreover,
    \[
        \operatorname{dist}(y,P_VK)
        \ge
        (|y|-R)_+
        =:
        \max \{ |y| - R ,0 \},
\]
and hence
    \[
        F_m(V,y)
        \le
        \exp\left(-((|y|-R)_+)^2\right).
    \]
We choose the radical dominating function $\mathcal{F}(|y|) = \exp(-[(|y| - R)_+]^2)$.
Thus, $F_m$ satisfies the radial domination assumption in
Theorem~\ref{thm:sec1-forward-grassmannian-bl}, and $f_{V,m}$ is an admissible input for
Theorem~\ref{thm:sec1-compact-grassmannian-bl}, with $\alpha(V)=R$.

For every $x\in K$ and $V\in G(n,k)$, we have
    \[
        P_Vx\in P_VK\subset B_V(0,R),
    \]
and therefore
    \[
        f_{V,m}(P_Vx)=1.
    \]
It follows 
    \[
        \mathbf{1}_K(x)
        \le
        \exp\left(
                \int_{G(n,k)}
                    \log f_{V,m}(P_Vx)
                \,d\mu(V)
            \right),
    \]
and hence
    \[
        |K|
        \le
        \int_{\R^n}
            \exp\left(
                \int_{G(n,k)}
                    \log f_{V,m}(P_Vx)
                \,d\mu(V)
            \right)
        \,dx.
    \]
Applying Theorem~\ref{thm:sec1-compact-grassmannian-bl}, we obtain
    \[
        |K|
        \le
        \exp\left(
            \frac{n}{k}
            \int_{G(n,k)}
                \log\left(
                    \int_{B_V(0,R)}
                        \exp\left(
                            -m\operatorname{dist}(y,P_VK)^2
                        \right)
                    \,d\lambda_V(y)
                \right)
            \,d\nu_{n,k}(V)
        \right).
    \]
For every $V\in G(n,k)$, dominated convergence gives
    \[
        \int_{B_V(0,R)}
            \exp\left(
                -m\operatorname{dist}(y,P_VK)^2
            \right)
        \,d\lambda_V(y)
        \longrightarrow
        |P_VK|
    \]
as $m\to\infty$.
Furthermore, Fubini's theorem gives the uniform lower bound
    \[
        |K|
        \le
        \omega_{n-k}R^{n-k}|P_VK|,
    \]
where $\omega_j:=|B_{\R^j}(0,1)|$.
Together with the upper bound $|P_VK|\le\omega_kR^k$, this allows us to apply dominated convergence to the logarithms.
We conclude that
    \begin{equation}
        |K|^{k/n}
        \le
        \exp\left(
            \int_{G(n,k)}
                \log|P_VK|
            \,d\nu_{n,k}(V)
        \right).
        \label{eq:intro-averaged-grassmannian-lw}
    \end{equation}
The corresponding statement for Lebesgue measurable sets follows from inner regularity.
Indeed, if $\Omega\subset\R^n$ is measurable and has positive finite measure, we may choose compact sets $K_j\subset\Omega$ such that $|K_j|\to|\Omega|$.
Since $P_VK_j\subset P_V\Omega$, inequality
\eqref{eq:intro-averaged-grassmannian-lw} implies
    \[
        |\Omega|^{k/n}
        \le
        \exp\left(
            \int_{G(n,k)}
                \log|P_V\Omega|
            \,d\nu_{n,k}(V)
        \right).
    \]
In particular, let $k=n-1$, and let $\sigma$ denote the rotation-invariant probability measure on $S^{n-1}$.
The map
    \[
        S^{n-1}\ni\omega\longmapsto\omega^\perp\in G(n,n-1)
    \]
pushes $\sigma$ forward to $\nu_{n,n-1}$.
Hence, taking logarithms in
    \eqref{eq:intro-averaged-grassmannian-lw}, we obtain
    \begin{equation}
        \frac{n-1}{n}\log|\Omega|
        \le
        \int_{S^{n-1}}
            \log|P_{\omega^\perp}\Omega|
        \,d\sigma(\omega).
        \label{eq:intro-logarithmic-averaged-lw}
    \end{equation}
This is precisely the logarithmic averaged Loomis--Whitney inequality appearing in
\cite[Remark~8.3]{BennettTao2024}.
By Jensen's inequality, it further implies
    \[
        |\Omega|^{(n-1)/n}
        \le
        \int_{S^{n-1}}
            |P_{\omega^\perp}\Omega|
        \,d\sigma(\omega),
    \]
which is the averaged Loomis--Whitney inequality stated as equation~(8.10) in that paper.

Inequality~\eqref{eq:intro-averaged-grassmannian-lw} is also related to quermassintegrals.
For a convex body $K\subset\R^n$, the $k$-th quermassintegral can be defined by Kubota's formula
    \[
        W_k(K)
        :=
        \frac{\omega_n}{\omega_k}
        \int_{G(n,k)}
            |P_VK|
        \,d\nu_{n,k}(V).
    \]
Following Milman and Yehudayoff
\cite[Definition~7.1]{MilmanYehudayoff2023}, one may more generally define the $L^p$-moment quermassintegrals by
    \[
        Q_{k,p}(K)
        :=
        \frac{\omega_n}{\omega_k}
        \left(
            \int_{G(n,k)}
                |P_VK|^p
            \,d\nu_{n,k}(V)
        \right)^{1/p},
        \qquad p\ne0.
    \]
The case $p=0$ is defined by taking limit as
    \[
        Q_{k,0}(K)
        :=
        \frac{\omega_n}{\omega_k}
        \exp\left(
            \int_{G(n,k)}
                \log|P_VK|
            \,d\nu_{n,k}(V)
        \right).
    \]
It follows immediately from
\eqref{eq:intro-averaged-grassmannian-lw} that
    \[
        Q_{k,0}(K)
        \ge
        \frac{\omega_n}{\omega_k}|K|^{k/n}.
    \]
Since the $L^p$-means are nondecreasing in $p$, we further obtain
    \begin{equation}
        Q_{k,p}(K)
        \ge
        Q_{k,0}(K)
        \ge
        \frac{\omega_n}{\omega_k}|K|^{k/n},
        \qquad p\ge0.
        \label{eq:intro-quermassintegral-lower-bound}
    \end{equation}
In particular,
    \[
        W_k(K)
        \ge
        \frac{\omega_n}{\omega_k}|K|^{k/n}.
    \]
The estimate in
\eqref{eq:intro-quermassintegral-lower-bound} is weaker than the sharp inequality of Milman and Yehudayoff
\cite[Theorems~1.3 and~7.2]{MilmanYehudayoff2023}.
If $B_K$ denotes the centered Euclidean ball satisfying $|B_K|=|K|$, their result states that
    \[
        Q_{k,p}(K)
        \ge
        Q_{k,p}(B_K)
        =
        \omega_n^{1-k/n}|K|^{k/n},
        \qquad p\ge-n.
    \]
For $1\le k<n$, one has
    \[
        \frac{\omega_n}{\omega_k}
        <
        \omega_n^{1-k/n},
    \]
so the constant obtained from our Grassmannian Brascamp--Lieb inequality is strictly smaller than the sharp constant.
For $p>-n$, equality in the sharp inequality holds if and only if $K$ is a Euclidean ball, while for $p=-n$ equality holds if and only if $K$ is an ellipsoid.

\subsection{A mixed-rank continuous Brascamp--Lieb inequality}
The equal-rank assumption in Theorem~\ref{thm:sec1-forward-grassmannian-bl} and~\ref{thm:sec1-compact-grassmannian-bl} is not essential.
Set
\begin{equation}
    \label{eqn:intro section-the mixed Grassmannian and tautological bundle}
    \mathcal{G}_n
    :=
    \bigsqcup_{k=1}^{n}
        G(n,k),
    \qquad
    \mathcal T_n
    :=
    \bigsqcup_{k=1}^{n}
        T_{n,k}.
\end{equation}
Thus, \(\mathcal{G}_n\) is the finite disjoint union of all Grassmannians in \(\R^n\), and the fibre of \(\mathcal T_n\) over \(V\in\mathcal{G}_n\) is the Euclidean space \(V\).
Let \(\mu\) be a finite Borel measure on \(\mathcal{G}_n\) satisfying
\begin{equation}
    \label{eqn:intro section-the mixed-rank isotropy condition}
    \int_{\mathcal{G}_n}
        P_V
    \,d\mu(V)
    =
    I_n.
\end{equation}
Taking traces in \eqref{eqn:intro section-the mixed-rank isotropy condition} gives the dimension balance
\begin{equation}
    \label{eqn:intro section-the mixed-rank dimension balance}
    \int_{\mathcal{G}_n}
        \dim(V)
    \,d\mu(V)
    =
    n.
\end{equation}
Let \(F:\mathcal T_n\to(0,\infty)\) be continuous, and suppose that, for every \(1\le k\le n\), there exists \(\mathcal F_k\in L^\infty([0,\infty))\) such that
\begin{equation}
    \label{eqn:intro section-the mixed-rank upper bounds}
    F(V,y)
    \le 
    \mathcal F_k(|y|)
    \quad\text{for \(V\in G(n,k)\) and \(y\in V\),}
    \qquad
    \int_0^\infty 
        \mathcal F_k(r)r^{k-1}
    \,d r
    <
    \infty.
\end{equation}
Then the corresponding mixed-rank continuous Brascamp--Lieb inequality is
\begin{equation}
    \label{eqn:intro section-the mixed-rank continuous Brascamp--Lieb inequality}
    \begin{aligned}
        &\int_{\R^n}
        \exp\left(
            \int_{\mathcal{G}_n}
                \log F(V,P_Vx)
            \,d\mu(V)
        \right)
        \,dx \\
        \le\,&
        \exp\left(
            \int_{\mathcal{G}_n}
            \log\left(
                \int_V F(V,y)
            \,d\lambda_V(y)
            \right)\,d\mu(V)
        \right).
    \end{aligned}
\end{equation}

The measure approximation proof of Theorem~\ref{thm:sec1-forward-grassmannian-bl} and~\ref{thm:sec1-compact-grassmannian-bl} extends directly to this setting. 
The heat-semigroup proof also extends without any essential change.
In the later sections, we will focus on the equal-rank case.
However, we will give a sketch of the approximate arguments and a sketch of the heat-flow proof to this mixed-rank inequality; see Section~\ref{sec:sec4-mixed-rank-bl}.

\begin{remark}
	\label{rem:sec1-cm-mixed-rank}
	The mixed-rank inequality also follows from an extension of Theorem~5.1 at the end of Section~5 of \cite{CorderoErausquinMaurey2017}.
	We consider only the non-truncated case $\alpha\equiv\infty$.

	In their work, the source function $f_s$ is defined on a Euclidean space $E_s$, and the target function $g_t$ is defined on a Euclidean space $F_t$.
	All these spaces are subspaces of a common Euclidean space $\R^\ell$.
	Their dimensions may depend on $s$ and $t$.
	The integral conclusion takes the form
	\begin{equation}
		\begin{aligned}
			\int_{\Omega_2}
			-\log\left(
				\int_{F_t}e^{-g_t(y)}\,d\lambda_{F_t}(y)
			\right)\,d\mu_2(t)
			\le
			\int_{\Omega_1}
			-\log\left(
				\int_{E_s}e^{-f_s(x)}\,d\lambda_{E_s}(x)
			\right)\,d\mu_1(s).
		\end{aligned}
		\label{eq:sec1-cm-variable-fibre-conclusion}
	\end{equation}
	Here, $\lambda_{E_s}$ and $\lambda_{F_t}$ denote the induced Lebesgue measures on the corresponding fibres.
	Let $\mu$ be the isotropic Borel measure in the above mixed-rank statement, so that
	\[
		\int_{\mathcal G_n}P_V\,d\mu(V)=I_n.
	\]
	We take $\ell=n$ and choose
	\[
		X_1=(\{\ast\},\delta_\ast),
		\qquad
		X_2=(\mathcal G_n,\mu),
	\]
	with the Borel $\sigma$-algebras.
	The source fibre and the target fibres are
	\[
		E_\ast=\R^n,
		\qquad
		F_V=V
		\quad\text{for }V\in\mathcal G_n.
	\]
	Thus, the target dimension is $\dim F_V=\dim V$ and may vary with $V$.
    The source section space is identified with $\R^n$.
	Define
    \[
    	A:\R^n\longrightarrow L^2(\mathcal G_n,\mu;\R^n),
    	\qquad
    	(Ax)(V):=P_Vx.
    \]
    The isotropy condition gives
    \[
    	\begin{aligned}
    	\|Ax\|_{L^2(\mathcal G_n,\mu;\R^n)}^2
    	&=
    	\int_{\mathcal G_n}|P_Vx|^2\,d\mu(V)\\
    	&=
    	\left\langle
    		\left(
    			\int_{\mathcal G_n}P_V\,d\mu(V)
    		\right)x,x
    	\right\rangle\\
    	&=
    	|x|^2.
    	\end{aligned}
    \]
    Thus, $A$ is an isometry.
    In the notation of their variable-fibre setting, the projections onto the source and target fibres are
	\[
		P(\ast)=I_n,
		\qquad
		Q(V)=P_V.
	\]
	The required compatibility condition therefore holds:
	\[
		A(s\mapsto P(s)w)(V)
		=
		(Aw)(V)
		=
		P_Vw
		=
		Q(V)w,
		\qquad
		w\in\R^n.
	\]
	Taking traces in the isotropy condition gives
	\[
		\int_{\mathcal G_n}\dim V\,d\mu(V)
		=
		\int_{\mathcal G_n}\operatorname{tr}(P_V)\,d\mu(V)
		=
		n.
	\]
	This is precisely the required dimension balance:
	\[
		\int_{\Omega_1}\dim E_s\,d\mu_1(s)
		=
		n
		=
		\int_{\Omega_2}\dim F_t\,d\mu_2(t).
	\]
    For the input in the mixed-rank inequality, set
	\[
		g_V(y):=-\log F(V,y),
		\qquad
		y\in V.
	\]
	Define the source function by
	\[
		\begin{aligned}
			f_\ast(x)
			:=
			\int_{\mathcal G_n}g_V(P_Vx)\,d\mu(V)
			=
			-\int_{\mathcal G_n}\log F(V,P_Vx)\,d\mu(V),
			\qquad
			x\in\R^n.
		\end{aligned}
	\]
	The functional hypothesis holds with equality:
	\[
		\int_{\mathcal G_n}g_V((Ax)(V))\,d\mu(V)
		=
		f_\ast(x)
		=
		\int_{\Omega_1}f_s(x)\,d\mu_1(s),
		\qquad
		x\in\R^n.
	\]
    The target fibre integrals satisfy
	\[
		\int_{F_V}e^{-g_V(y)}\,d\lambda_{F_V}(y)
		=
		\int_VF(V,y)\,d\lambda_V(y).
	\]
	On the source side,
	\[
		\int_{E_\ast}e^{-f_\ast(x)}\,d\lambda_{E_\ast}(x)
		=
		\int_{\R^n}
		\exp\left(
			\int_{\mathcal G_n}\log F(V,P_Vx)\,d\mu(V)
		\right)\,d x.
	\]
	Substituting these identities into \eqref{eq:sec1-cm-variable-fibre-conclusion} gives
	\[
		\begin{aligned}
			-\int_{\mathcal G_n}
			\log\left(
				\int_VF(V,y)\,d\lambda_V(y)
			\right)\,d\mu(V)
			\le
			-\log\left[
				\int_{\R^n}
				\exp\left(
					\int_{\mathcal G_n}\log F(V,P_Vx)\,d\mu(V)
				\right)\,d x
			\right].
		\end{aligned}
	\]
	Multiplying both sides by $-1$ and taking exponentials gives
	\[
		\begin{aligned}
			\int_{\R^n}
			\exp\left(
				\int_{\mathcal G_n}\log F(V,P_Vx)\,d\mu(V)
			\right)\,d x
			\le
			\exp\left(
				\int_{\mathcal G_n}
				\log\left(
					\int_VF(V,y)\,d\lambda_V(y)
				\right)\,d\mu(V)
			\right).
		\end{aligned}
	\]
	This is the mixed-rank continuous Grassmannian Brascamp--Lieb inequality.
\end{remark}

\subsection{Organization of the paper}
\label{subsec:sec1-paper-organization}

The remainder of the paper is organized as follows.
In Section~\ref{sec:sec2-approximation-proofs}, we prove Theorem~\ref{thm:sec1-forward-contin-rank-one-bl}, Theorem~\ref{thm:sec1-forward-grassmannian-bl}, and Theorem~\ref{thm:sec1-compact-grassmannian-bl} by measure approximation.
In Section~\ref{sec:sec3-heat-semigroup-proof}, we establish Theorem~\ref{thm:sec1-heat-semigroup-preservation} and derive the Grassmannian inequalities by heat flow.
In Section~\ref{sec:sec4-mixed-rank-bl}, we state a mixed-rank continuous Brascamp--Lieb inequality and give sketches of proof by measure approximation and heat-semigroup preservation.

\section{Approximation Proofs for the Continuous Brascamp--Lieb Inequalities}
\label{sec:sec2-approximation-proofs}

The approximation method used in this section has two steps.
We first approximate a continuous isotropic measure by discrete measures that satisfy the isotropic condition.
We then apply the corresponding geometric Brascamp--Lieb inequality and pass to the weak limit.
On the sphere, the construction of approximate discrete isotropic measures is standard.
For higher-dimensional subspaces we use a Tchakaloff-type theorem~\cite{Tchakaloff1957,SchaeferUllrich2025}.

We present the main approximation arguments before proving all of the continuity and measurability statements needed in the Grassmannian case.
This order makes the passage from the discrete inequality to the continuous one visible without interrupting it with technical details.

\subsection{The Rank-One Case}
\label{subsec:sec2-rank-one-case}

We begin with a standard discretization fact for finite measures on compact metric spaces.

\begin{lemma}
    \label{lem:sec2-compact-metric-discretization}
    Let $(X,d)$ be a compact metric space and let $\sigma$ be a finite positive Borel measure on $X$.
    For every $\varepsilon>0$, there are points $x_1^\varepsilon,\ldots,x_{m(\varepsilon)}^\varepsilon$ and a (disjoint) Borel partition $(U_i^\varepsilon)_{i=1}^{m(\varepsilon)}$ of $X$ such that
	\[
	    x_i^\varepsilon
	    \in
	    U_i^\varepsilon
	    \subset
	    \overline B_d(x_i^\varepsilon,\varepsilon).
	\]
    Here,
	\[
	    \overline B_d(x_i^\varepsilon,\varepsilon)
	    :=
	    \{x\in X:d(x,x_i^\varepsilon)\le\varepsilon\}
	\]
    denotes the closed metric ball in $(X,d)$.
    Moreover, the discrete measures
	\begin{equation}
	    \sigma_\varepsilon
	    :=
	    \sum_{i=1}^{m(\varepsilon)}
	        \sigma(U_i^\varepsilon)\delta_{x_i^\varepsilon}
	    \label{eq:sec2-compact-discrete-measure}
	\end{equation}
    converge weakly to $\sigma$ as $\varepsilon\downarrow0$.
\end{lemma}

\begin{proof}
    Compactness implies total boundedness, so every $\varepsilon$-separated subset of $X$ is finite.
    We may therefore choose a finite maximal $\varepsilon$-separated set $\{x_i^\varepsilon\}_{i=1}^{m(\varepsilon)}$.
    If some $x\in X$ had distance greater than $\varepsilon$ from every $x_i^\varepsilon$, then adjoining $x$ would produce a larger $\varepsilon$-separated set.
    Maximality consequently gives
	\[
	    X
	    =
	    \bigcup_{i=1}^{m(\varepsilon)}
	        \overline B_d(x_i^\varepsilon,\varepsilon).
	\]

    To construct the partition explicitly, define
	\[
	    E_i^\varepsilon
	    :=
	    \left\{
	        x\in X:
	        d(x,x_i^\varepsilon)
	        =
	        \min_{1\le j\le m(\varepsilon)}d(x,x_j^\varepsilon)
	    \right\}
	\]
    and resolve ties in the order of the indices by setting
	\[
	    U_1^\varepsilon:=E_1^\varepsilon,
	    \qquad
	    U_i^\varepsilon
	    :=
	    E_i^\varepsilon\setminus\bigcup_{j<i}E_j^\varepsilon
	    \quad
	    \text{for }i\ge2.
	\]
    Every $E_i^\varepsilon$ is Borel because the distance functions are continuous, so every $U_i^\varepsilon$ is Borel.
    These sets are disjoint and cover $X$.
    The separation of the centers implies $x_i^\varepsilon\in U_i^\varepsilon$, and the covering property implies $U_i^\varepsilon\subset\overline B_d(x_i^\varepsilon,\varepsilon)$.

    For $g\in C(X)$, set
	\[
	    \omega_g(t)
	    :=
	    \sup
	        \bigl\{|g(x)-g(y)|:d(x,y)\le t\bigr\}.
	\]
    Uniform continuity of $g$ gives $\omega_g(t)\to0$ as $t\downarrow0$.
    By the definition of $\sigma_\varepsilon$ and the disjointness of the partition,
	\begin{align*}
	    \int_Xg\,d\sigma-
	    \int_Xg\,d\sigma_\varepsilon
	    &=
	    \sum_{i=1}^{m(\varepsilon)}
	        \int_{U_i^\varepsilon}
	            \bigl(g(x)-g(x_i^\varepsilon)\bigr)
	        \,d\sigma(x).
	\end{align*}
    Hence
	\begin{equation}
	    \left|
	        \int_X
	            g
	        \,d\sigma
	        -
	        \int_X
	            g
	        \,d\sigma_\varepsilon
	    \right|
	    \le
	    \sigma(X)\omega_g(\varepsilon).
	    \label{eq:sec2-compact-weak-error}
	\end{equation}
    The right-hand side tends to zero, which is precisely $\sigma_\varepsilon\rightharpoonup\sigma$.
\end{proof}

Following the steps in Barthe \cite{Barthe2004}, let $\mu$ be an isotropic measure on $S^{n-1}$, so that
	\begin{equation}
	    \int_{S^{n-1}}
	        u\otimes u
	    \,d\mu(u)
	    =
	    I_n.
	    \label{eq:sec2-spherical-isotropy}
	\end{equation}
Taking traces in \eqref{eq:sec2-spherical-isotropy} gives
	\begin{equation}
	    \mu(S^{n-1})
	    =
	    n.
	    \label{eq:sec2-spherical-total-mass}
	\end{equation}
Apply Lemma~\ref{lem:sec2-compact-metric-discretization} to $S^{n-1}$ and write
	\[
	    \mu_\varepsilon
	    =
	    \sum_i
	        \mu(U_i^\varepsilon)\delta_{u_i^\varepsilon}.
	\]
Here $\{u^\varepsilon_i\}$ is the chosen finite maximal $\varepsilon$-net and $\{U^\varepsilon_i\}$ is the corresponding Borel partition, for any fixed $\varepsilon > 0$.
This measure need not be isotropic.
However, it is approximately isotropic.
Indeed, we may compute the following matrix
	\begin{equation}
	    Q_\varepsilon
	    :=
	    \int_{S^{n-1}}
	        u\otimes u
	    \,d\mu_\varepsilon(u)
	    =
	    \sum_i\mu(U_i^\varepsilon)
	        u_i^\varepsilon\otimes u_i^\varepsilon.
	    \label{eq:sec2-rank-one-isotropy-matrix}
	\end{equation}
For unit vectors $u$ and $v$,
	\[
	    \|u\otimes u-v\otimes v\|_{\mathrm{op}}
	    \le 2|u-v|.
	\]
Consequently, since every $u\in U_i^\varepsilon$ satisfies $|u-u_i^\varepsilon|\le\varepsilon$,
	\begin{equation}
	    \|Q_\varepsilon-I_n\|_{\mathrm{op}}
	    \le
	    \sum_i
	    \int_{U_i^\varepsilon}
	        \|u_i^\varepsilon\otimes u_i^\varepsilon-u\otimes u\|_{\mathrm{op}}
	    \,d\mu(u)
	    \le
	    2\varepsilon\mu(S^{n-1})
	    =
	    2n\varepsilon.
        \label{eq:sec2-rank-one-operator-convergence}
	\end{equation}
Thus $Q_\varepsilon\to I_n$ in operator norm.
If $\|Q_\varepsilon-I_n\|_{\mathrm{op}}<1/2$, then every eigenvalue of $Q_\varepsilon$ lies in $[1/2,3/2]$, so $Q_\varepsilon$ is positive definite for sufficiently small $\varepsilon$.
The scalar function $s\mapsto s^{-1/2}$ is Lipschitz on $[1/2,3/2]$.
The spectral theorem therefore gives a constant $C<\infty$ such that
	\[
	    \|Q_\varepsilon^{-1/2}-I_n\|_{\mathrm{op}}
	    \le
	    C\|Q_\varepsilon-I_n\|_{\mathrm{op}}
	\]
for all sufficiently small $\varepsilon$.
Consequently,
	\begin{equation}
	    Q_\varepsilon^{-1/2}
	    \longrightarrow
	    I_n
	    \quad\text{in operator norm as }\varepsilon \to 0.
	    \label{eq:sec2-rank-one-inverse-square-root}
	\end{equation}

For such $\varepsilon$, define
	\[
	    T_\varepsilon(u)
	    :=
	    \frac{Q_\varepsilon^{-1/2}u}
	         {|Q_\varepsilon^{-1/2}u|},
	    \qquad u\in S^{n-1}.
	\]
If $B_\varepsilon=Q_\varepsilon^{-1/2}$, then, for small $\varepsilon$,
	\begin{align*}
	    |T_\varepsilon(u)-u|
	    &=
	    \frac{\big|B_\varepsilon u-|B_\varepsilon u|u\big|}{|B_\varepsilon u|}\\
	    &\le
	    \frac{|B_\varepsilon u-u|+\bigl||B_\varepsilon u|-1\bigr|}{1-\|B_\varepsilon-I_n\|_{\mathrm{op}}}\\
	    &\le
	    \frac{2\|B_\varepsilon-I_n\|_{\mathrm{op}}}{1-\|B_\varepsilon-I_n\|_{\mathrm{op}}}.
	\end{align*}
Taking the supremum over $u\in S^{n-1}$ and using \eqref{eq:sec2-rank-one-inverse-square-root} shows that $T_\varepsilon\to\operatorname{Id}_{S^{n-1}}$ uniformly.
We correct both the locations and the weights by setting
	\begin{equation}
	    \nu_\varepsilon
	    :=
	    \sum_i
	        \mu(U_i^\varepsilon)
	        |Q_\varepsilon^{-1/2}u_i^\varepsilon|^2
	        \delta_{T_\varepsilon(u_i^\varepsilon)}.
	    \label{eq:sec2-rank-one-corrected-measure}
	\end{equation}
The use of $\mu(U_i^\varepsilon)$ in \eqref{eq:sec2-rank-one-corrected-measure} is essential.
Direct computation gives
	\begin{align*}
	    \int_{S^{n-1}}
	        u\otimes u
	    \,d\nu_\varepsilon(u)
	    &=
	    \sum_i\mu(U_i^\varepsilon)
	      (Q_\varepsilon^{-1/2}u_i^\varepsilon)
	      \otimes(Q_\varepsilon^{-1/2}u_i^\varepsilon)\\
	    &=
	    Q_\varepsilon^{-1/2}Q_\varepsilon Q_\varepsilon^{-1/2}
	    =
	    I_n.
	\end{align*}
This standard construction is adapted from Barthe's argument~\cite{Barthe2004}.
We include it for completeness.
We record the construction as a theorem.

\begin{theorem}
    \label{thm:sec2-spherical-isotropic-approximation}
    Every finite isotropic Borel measure on $S^{n-1}$ is the weak limit of finite discrete isotropic measures on $S^{n-1}$.
\end{theorem}

\begin{proof}
    It remains only to verify weak convergence of the corrected measures.
    For $g\in C(S^{n-1})$, let
	\[
	    r_\varepsilon
	    :=
	    \sup_{u\in S^{n-1}}|T_\varepsilon(u)-u|.
	\]
    Put $a_i^\varepsilon:=|Q_\varepsilon^{-1/2}u_i^\varepsilon|^2$.
    The isotropy of $\nu_\varepsilon$ gives $\sum_i\mu(U_i^\varepsilon)a_i^\varepsilon=n$.
    Moreover,
	\[
	    \sup_i|a_i^\varepsilon-1|
	    \le
	    \|Q_\varepsilon^{-1}-I_n\|_{\mathrm{op}}
	    \longrightarrow0.
	\]
    Decomposing
	\[
	    a_i^\varepsilon g(T_\varepsilon(u_i^\varepsilon))-g(u_i^\varepsilon)
	    =
	    a_i^\varepsilon\bigl(g(T_\varepsilon(u_i^\varepsilon))-g(u_i^\varepsilon)\bigr)
	    +(a_i^\varepsilon-1)g(u_i^\varepsilon)
	\]
    and using \eqref{eq:sec2-spherical-total-mass}, we obtain
	\begin{align*}
	    \left|
	        \int
	            g
	        \,d\nu_\varepsilon
	        -
	        \int
	            g
	        \,d\mu_\varepsilon
	    \right|
	    \le
	    n\biggl(&
	        \|g\|_\infty
	        \sup_{u\in S^{n-1}}
	            \bigl||Q_\varepsilon^{-1/2}u|^2-1\bigr|
	        +\omega_g(r_\varepsilon)
	    \biggr).
	\end{align*}
    Both terms tend to zero by
    \eqref{eq:sec2-rank-one-inverse-square-root} and the uniform continuity of $g$.
    Since $\mu_\varepsilon\rightharpoonup\mu$ by Lemma~\ref{lem:sec2-compact-metric-discretization}, the triangle inequality gives $\nu_\varepsilon\rightharpoonup\mu$.
\end{proof}

We next apply this approximation to the geometric rank-one Brascamp--Lieb inequality.
Recall that if $c_i>0$, $u_i\in S^{n-1}$, and $\sum_i c_i u_i\otimes u_i=I_n$, then
	\begin{equation}
	    \int_{\R^n}
	        \prod_i
	            h_i(\langle x,u_i\rangle)^{c_i}
	    \,dx
	    \le
	    \prod_i
	        \left(
	            \int_\R
	                h_i(s)
	            \,ds
	        \right)^{c_i}
	    \label{eq:sec2-rank-one-geometric-bl}
	\end{equation}
for nonnegative integrable inputs; see, for example, \cite{Ball1989,Barthe2004}.

\begin{proof}[Proof of Theorem~\ref{thm:sec1-forward-contin-rank-one-bl}]
    Let $\nu_\varepsilon$ be given by Theorem~\ref{thm:sec2-spherical-isotropic-approximation}.
    For a fixed $x\in\R^n$, the function
	\[
	    u
	    \longmapsto
	    \log F(u,\langle x,u\rangle)
	\]
    is continuous on $S^{n-1}$.
    Also, dominated convergence with the integrable domination $s\mapsto\mathcal F(s)$ shows that
	\[
	    u
	    \longmapsto
	    J(u):=\int_\R F(u,s)\,ds
	\]
    is continuous.
    It is strictly positive, so $\log J$ is continuous and hence bounded on the sphere.
    Weak convergence therefore gives
	\begin{align*}
	    \int\log F(u,\langle x,u\rangle)\,d\nu_\varepsilon(u)
	        &\longrightarrow
	        \int\log F(u,\langle x,u\rangle)\,d\mu(u),\\
	    \int\log J(u)\,d\nu_\varepsilon(u)
	        &\longrightarrow
	        \int\log J(u)\,d\mu(u).
	\end{align*}
    On the other hand, the integrands on the left-hand side of~\eqref{eq:sec1-forward-contin-rank-one-bl}
        \[
            x\mapsto
            \exp
            \left(
            \int_{S^{n-1}}
                \log
                    F(u,\inner{x}{u})
            \,d\mu(u)
            \right)
        \]
    is nonnegative and Lebesgue measurable on $\R^n$.
    Indeed, let $L(x,u) = \log F(u,\inner{x}{u})$.
    Since $F > 0$ is continuous, $L$ is continuous on $S^{n-1}\times \R^n$.
    We split $L$ into positive and negative parts,
        \[
            L = L^+ - L^- := \max\{L,0\} - \max\{-L,0\}.
        \]
    Then both $L^+$ and $L^-$ are nonnegative and measurable.
    By Fubini--Tonelli theorem, 
        \[
            x
            \mapsto
            \int_{S^{n-1}}
                L^+(x,u)
            \,d\mu(u)
            \quad
            \text{and}
            \quad
            x
            \mapsto
            \int_{S^{n-1}}
                L^-(x,u)
            \,d\mu(u)
        \]
    are Lebesgue measurable. Also, since $0\le f_u \le \mathcal{F}\in L^1\cap L^\infty$, we have
        \[
            \int_{S^{n-1}}
                L^+(x,u)
            \,d\mu(u)
            <
            \infty.
        \]
    there is no $\infty - \infty$ ambiguity happens in
        \[
            x
            \mapsto
            \int_{S^{n-1}}
                L(x,u)
            \,d\mu(u)
            =
            \int_{S^{n-1}}
                L^+(x,u)
            \,d\mu(u)
            -
            \int_{S^{n-1}}
                L^-(x,u)
            \,d\mu(u).
        \]

    After discarding zero coefficients, write $\nu_\varepsilon=\sum_i c_i^\varepsilon \delta_{u_i^\varepsilon}$ with $c_i^\varepsilon>0$.
    Fatou's lemma applies to the nonnegative measurable left-hand integrands, and the geometric Brascamp--Lieb inequality~\eqref{eq:sec2-rank-one-geometric-bl} applies because $\sum_i c_i^\varepsilon u_i^\varepsilon\otimes u_i^\varepsilon=I_n$.
    Thus
    	\begin{align*}
    	    \int_{\R^n}
    	        \exp\left(
    	            \int\log F(u,\langle x,u\rangle)\,d\mu(u)
    	        \right)dx
    	    &\le\,
    	    \liminf_{\varepsilon\downarrow0}
    	        \int_{\R^n}
    	        \exp\left(
    	            \int\log F(u,\langle x,u\rangle)\,d\nu_\varepsilon(u)
    	        \right)dx\\
    	    \le\,
    	    \lim_{\varepsilon\downarrow0}
    	        \exp\left(
    	            \int\log J(u)\,d\nu_\varepsilon(u)
    	        \right)
    	    &=\,
    	        \exp\left(
    	            \int
                        \log 
                            \Big(
                                \int_\R
                                    F(u,s)  
                                \,ds
                            \Big)
                    \,d\mu(u)
    	        \right).
    	\end{align*}
    This proves Theorem~\ref{thm:sec1-forward-contin-rank-one-bl} when $\alpha\equiv-\infty$ and $\beta\equiv+\infty$.
    The extension to the case in which one or both endpoints $\alpha$ and $\beta$ are finite follows from the same cutoff and regularization argument used later in the proof of the Grassmannian inequality.
We therefore omit the details here.
\end{proof}

\subsection{Discrete Isotropic Approximation on the Grassmannian}
\label{subsec:sec2-grassmannian-discrete-approximation}

Let $G(n,k)$ be the Grassmannian of $k$-dimensional linear subspaces of $\R^n$.
We identify each subspace $V$ with its orthogonal projection $P_V$, regarded as a linear operator on $\R^n$, and use the metric
	\begin{equation}
	    d(V,W):=\|P_V-P_W\|_{\mathrm{op}}.
	    \label{eq:sec2-grassmannian-metric}
	\end{equation}
This makes $G(n,k)$ a compact metric space.
A finite Borel measure $\mu$ on $G(n,k)$ is \textit{isotropic} when
	\begin{equation}
	    \int_{G(n,k)}P_V\,d\mu(V)=I_n.
	    \label{eq:sec2-grassmannian-isotropy}
	\end{equation}
Since $\tr P_V=k$, taking traces gives
	\begin{equation}
	    \mu(G(n,k))=\frac nk.
	    \label{eq:sec2-grassmannian-total-mass}
	\end{equation}

For every $\varepsilon>0$, choose a finite maximal $\varepsilon$-net $\{V_i^\varepsilon\}_i$.
As in Lemma~\ref{lem:sec2-compact-metric-discretization}, it determines a Borel partition
	\begin{equation}
	    G(n,k)=\bigsqcup_i U_i^\varepsilon,
	    \qquad
	    U_i^\varepsilon\subset
	        \overline B_d(V_i^\varepsilon,\varepsilon).
	    \label{eq:sec2-grassmannian-small-cell-partition}
	\end{equation}
Here, $\overline B_d(V_i^\varepsilon,\varepsilon)$ denotes the closed ball in the topology induced by the metric $d$ centered at $V^\varepsilon_i$ with radius $\varepsilon$.

\begin{remark}
    The rank-one correction does not directly extend to $k>1$.
    If $Q^{-1/2}$ is applied to a $k$-plane $V$, then in general $Q^{-1/2}P_VQ^{-1/2}$ is not a scalar multiple of the orthogonal projection onto $Q^{-1/2}V$.
    The distortion of $Q^{-1/2}|_V$ is anisotropic and cannot be absorbed into a single positive coefficient.
    This is exactly what was possible on a one-dimensional subspace, where every positive quadratic form is scalar.
    More explicitly, the nonzero eigenvalues of $Q^{-1/2}P_VQ^{-1/2}$ are the squares of the singular values of $Q^{-1/2}|_V$.
    If this operator were equal to $cP_W$ for some $c>0$ and $W\in G(n,k)$, then all of its nonzero eigenvalues would equal $c$.
    For $k>1$, the restriction $Q^{-1/2}|_V$ generally has unequal eigenvalues, so this necessary condition fails.
    We therefore preserve the projection identity locally instead of correcting it globally.
    We use the following Tchakaloff-type theorem in the form needed here.
\end{remark}

\begin{theorem}
    \label{thm:sec2-tchakaloff-quadrature}
    Let $(\Omega,\mathcal M,\sigma)$ be a finite positive measure space, and let $X$ be a finite-dimensional space of real-valued $\sigma$-integrable functions on $\Omega$ containing the constant function $1$.
    There are at most $\dim X$ points $\omega_\ell\in\Omega$ and nonnegative coefficients $c_\ell$ such that
	\begin{equation}
	    \int_\Omega h\,d\sigma
	    =\sum_\ell c_\ell h(\omega_\ell),
	    \qquad h\in X.
	    \label{eq:sec2-tchakaloff-exactness}
	\end{equation}
\end{theorem}

We use the generalized Tchakaloff theorem in the form established by Sch\"afer and Ullrich~\cite[Theorem~1.1]{SchaeferUllrich2025}.
We apply it separately on small Borel cells.

\begin{theorem}
    \label{thm:sec2-grassmannian-isotropic-approximation}
    Every isotropic Borel measure $\mu$ on $G(n,k)$ is the weak limit of finite discrete isotropic measures on $G(n,k)$.
\end{theorem}

\begin{proof}
    Fix $\varepsilon>0$ and use the partition in~\eqref{eq:sec2-grassmannian-small-cell-partition}.
    On every cell of positive $\mu$-measure, apply Theorem~\ref{thm:sec2-tchakaloff-quadrature} to the restrictions of the functions in
	\[
	    X
	    :=
	    \operatorname{span}
	    \left(
	        \{1\}\cup
	        \{V\mapsto\langle P_Ve_a,e_b\rangle:1\le a\le b\le n\}
	    \right).
	\]
    Here $\{e_1,\dots,e_n\}$ is the standard basis of $\R^n$.
    All these functions are bounded and Borel measurable.
    The dimension of $X$ is at most $1+n(n+1)/2$, so the number of quadrature nodes on each cell is bounded independently of $\varepsilon$ and of the cell.
    We obtain points $V_{i,\ell}^\varepsilon\in U_i^\varepsilon$ and nonnegative weights $c_{i,\ell}^\varepsilon$ such that exactness for the constant function gives
	\begin{equation}
	    \sum_\ell c_{i,\ell}^\varepsilon
	    =\mu(U_i^\varepsilon),
	    \label{eq:sec2-local-mass-exactness}
	\end{equation}
    whereas exactness for the matrix entries gives
	\begin{equation}
	    \sum_\ell c_{i,\ell}^\varepsilon
	        P_{V_{i,\ell}^\varepsilon}
	    =
	    \int_{U_i^\varepsilon}P_V\,d\mu(V).
	    \label{eq:sec2-local-projection-exactness}
	\end{equation}
    The restriction $a\le b$ is sufficient because every $P_V$ is symmetric.
    If $\mu(U_i^\varepsilon)=0$, we use no node on that cell; both identities then hold with empty sums.

    Define the local and global discrete measures by
	\begin{equation}
	    \mu_i^\varepsilon
	    :=\sum_\ell c_{i,\ell}^\varepsilon
	        \delta_{V_{i,\ell}^\varepsilon},
	    \qquad
	    \nu_\varepsilon:=\sum_i\mu_i^\varepsilon.
	    \label{eq:sec2-grassmannian-discrete-measure}
	\end{equation}
    Summing \eqref{eq:sec2-local-projection-exactness} over the cells and using \eqref{eq:sec2-grassmannian-isotropy}, we find
	\begin{equation}
	    \int_{G(n,k)}P_V\,d\nu_\varepsilon(V)
	    =I_n.
	    \label{eq:sec2-discrete-grassmannian-isotropy}
	\end{equation}
    After zero coefficients are discarded, all weights are positive real numbers.

    It remains to prove weak convergence.
    If $g\in C(G(n,k))$, then any two points in the same cell $U_i^\varepsilon$ are at distance at most $2\varepsilon$.
    By \eqref{eq:sec2-local-mass-exactness}, both the integral of $g$ over a cell and its quadrature sum lie between the cell mass times the minimum and the maximum of $g$ on that cell.
    Hence
	\begin{align}
	    \left|
	        \int g\,d\mu-
	        \int g\,d\nu_\varepsilon
	    \right|
	    &\le
	    \sum_i\mu(U_i^\varepsilon)
	        \operatorname{osc}_{U_i^\varepsilon}g\notag\\
	    &\le
	    \mu(G(n,k))\omega_g(2\varepsilon)
	    =\frac nk\,\omega_g(2\varepsilon)
	    \longrightarrow0.
	    \label{eq:sec2-grassmannian-weak-error}
	\end{align}
    Here $\omega_g$ is the modulus of continuity with respect to \eqref{eq:sec2-grassmannian-metric}; it tends to zero because the Grassmannian is compact, and we write
	\[
	    \operatorname{osc}_{U_i^\varepsilon}g
	    :=
	    \sup_{V\in U_i^\varepsilon}g(V)
	    -
	    \inf_{V\in U_i^\varepsilon}g(V).
	\]
    Thus $\nu_\varepsilon\rightharpoonup\mu$.
\end{proof}

\subsection{Formal Passage to the Continuous Inequality}
\label{subsec:sec2-formal-continuous-passage}

We now give the main approximation argument.
Assuming for the moment the continuity and measurability statements established in Subsection~\ref{subsec:sec2-continuity-measurability}, we use the following three facts: for each fixed $x\in\R^n$,
	\[
	    V
	    \longmapsto
	    \log F(V,P_Vx)
	\]
is continuous; the function
	\[
	    V
	    \longmapsto
	    \log
	        \left(
	            \int_V
	                F(V,y)
	            \,d\lambda_V(y)
	        \right)
	\]
is continuous; and the resulting left-hand integrands on $\R^n$ are Lebesgue measurable.
We now prove the main Grassmannian inequality stated in the Introduction.

\begin{proof}[Proof of Theorem~\ref{thm:sec1-forward-grassmannian-bl}]
    Let $\nu_\varepsilon$ be the discrete isotropic measures supplied by Theorem~\ref{thm:sec2-grassmannian-isotropic-approximation}.
    After reindexing and discarding zero weights, write
	\begin{equation}
	    \nu_\varepsilon
	    =\sum_{i=1}^{N(\varepsilon)}
	        c_i^\varepsilon\delta_{V_i^\varepsilon},
	    \qquad
	    c_i^\varepsilon>0,
	    \qquad
	    \sum_i c_i^\varepsilon P_{V_i^\varepsilon}=I_n.
	    \label{eq:sec2-discrete-equal-rank-datum}
	\end{equation}
    The equal-rank geometric Brascamp--Lieb inequality~\cite{BCCT} applied to~\eqref{eq:sec2-discrete-equal-rank-datum} gives
	\begin{align}
	    &\int_{\R^n}
	        \exp\left(
	            \int_{G(n,k)}
	                \log F(V,P_Vx)\,d\nu_\varepsilon(V)
	        \right)dx\notag\\
	    \le\,&
	        \exp\left(
	            \int_{G(n,k)}
	                \log\left(
	                    \int_VF(V,y)\,d\lambda_V(y)
	                \right)d\nu_\varepsilon(V)
	        \right).
	    \label{eq:sec2-discrete-equal-rank-bl}
	\end{align}

    For each fixed $x$, the weak convergence $\nu_\varepsilon\rightharpoonup\mu$ implies convergence of the logarithm in the left-hand integrand.
    Fatou's lemma is applicable because the exponentials are nonnegative and, by the measurability statement deferred to Subsection~\ref{subsec:sec2-continuity-measurability}, Lebesgue measurable.
    Therefore,
	\begin{align*}
	    &\int_{\R^n}
	        \exp\left(
	            \int\log F(V,P_Vx)\,d\mu(V)
	        \right)dx\\
	    \le\,&
	    \liminf_{\varepsilon\downarrow0}
	        \int_{\R^n}
	        \exp\left(
	            \int\log F(V,P_Vx)\,d\nu_\varepsilon(V)
	        \right)dx.
	\end{align*}
    On the right-hand side in~\eqref{eq:sec2-discrete-equal-rank-bl}, the logarithm part is a continuous, hence bounded, test function on the compact Grassmannian.
    Weak convergence consequently gives
	\[
	    \int
	        \log\left(\int_VF(V,y)\,d\lambda_V(y)\right)
	        d\nu_\varepsilon(V)
	    \longrightarrow
	    \int
	        \log\left(\int_VF(V,y)\,d\lambda_V(y)\right)
	        d\mu(V).
	\]
    Combining this limit with \eqref{eq:sec2-discrete-equal-rank-bl} proves \eqref{eq:sec1-forward-grassmannian-bl}.
\end{proof}

\subsection{Continuity and Measurability}
\label{subsec:sec2-continuity-measurability}

We now establish the analytic statements used in the proof of Theorem~\ref{thm:sec1-forward-grassmannian-bl}.
The first observation follows immediately from the projection metric.

\begin{lemma}
    \label{lem:sec2-projected-input-continuity}
    For every fixed $x\in\R^n$, the map $V\mapsto P_Vx$ is continuous from $G(n,k)$ to $\R^n$.
    Consequently,
	\[
	    V\longmapsto\log F(V,P_Vx)
	\]
    is continuous on $G(n,k)$.
\end{lemma}

\begin{proof}
    If $V_j\to V$, then
	\[
	    |P_{V_j}x-P_Vx|
	    \le\|P_{V_j}-P_V\|_{\mathrm{op}}|x|
	    \longrightarrow0.
	\]
    Thus $(V_j,P_{V_j}x)\to(V,P_Vx)$ in $T_{n,k}$.
    The second assertion follows from the continuity and strict positivity of $F$.
\end{proof}

To compare integrals over nearby fibers, we require isometries that vary continuously with the subspace.

\begin{lemma}
    \label{lem:sec2-nearby-subspace-isometry}
    Suppose that $V_j\to V$ in $G(n,k)$.
    For all sufficiently large $j$, define
	\begin{equation}
	    L_j:=P_{V_j}|_V:V\to V_j,
	    \qquad
	    A_j:=L_j^*L_j=P_VP_{V_j}|_V,
	    \qquad
	    O_j:=L_jA_j^{-1/2}.
	    \label{eq:sec2-nearby-subspace-polar-map}
	\end{equation}
    Then $O_j:V\to V_j$ is a linear isometric isomorphism and
	\begin{equation}
	    \|O_j-I_V\|_{V\to\R^n}\longrightarrow0,
	    \label{eq:sec2-nearby-isometry-convergence}
	\end{equation}
    where $I_V$ denotes the inclusion of $V$ into $\R^n$.
    Moreover, the pushforward measure
	\begin{equation}
	    (O_j)_\#\lambda_V=\lambda_{V_j}.
	    \label{eq:sec2-isometry-pushforward-measure}
	\end{equation}
\end{lemma}

\begin{proof}
    For $y\in V$,
	\[
	    |L_jy-y|
	    =|(P_{V_j}-P_V)y|
	    \le\|P_{V_j}-P_V\|_{\mathrm{op}}|y|.
	\]
    Thus $L_j$ is injective, and hence bijective, for all sufficiently large $j$.
    Its adjoint is $L_j^*=P_V|_{V_j}$, so the identity for $A_j$ in~\eqref{eq:sec2-nearby-subspace-polar-map} is
	\[
	    L_j^*L_j=P_VP_{V_j}|_V.
	\]
    Moreover,
	\[
	    \|A_j-I_V\|_{V\to V}
	    \le\|P_{V_j}-P_V\|_{\mathrm{op}}
	    \longrightarrow0.
	\]
    The operators $A_j$ are therefore positive definite for large $j$, and the continuous functional calculus gives $A_j^{-1/2}\to I_V$ in operator norm.
    Since
	\[
	    O_j^*O_j
	    =A_j^{-1/2}L_j^*L_jA_j^{-1/2}=I_V,
	\]
    $O_j$ is an isometry; equality of dimensions makes it onto $V_j$.
    Finally,
	\begin{align*}
	    \|O_j-I_V\|_{V\to\R^n}
	    &\le
	    \|L_j\|\,\|A_j^{-1/2}-I_V\|
	    +\|L_j-I_V\|_{V\to\R^n}
	    \longrightarrow0,
	\end{align*}
    which proves \eqref{eq:sec2-nearby-isometry-convergence}.

    To prove \eqref{eq:sec2-isometry-pushforward-measure}, choose linear isometries $U:\R^k\to V$ and $S:\R^k\to V_j$.
    The map $S^{-1}O_jU$ is orthogonal on $\R^k$ and hence preserves $k$-dimensional Lebesgue measure.
    Since $\lambda_V=U_\#\lambda_{\R^k}$ and $\lambda_{V_j}=S_\#\lambda_{\R^k}$, the pushforward identity follows.
\end{proof}

\begin{lemma}
    \label{lem:sec2-continuity-fiber-integral}
    Under \eqref{eq:sec1-grassmannian-radial-majorant}, the map
	\begin{equation}
	    V\longmapsto
	    J(V):=\int_VF(V,y)\,d\lambda_V(y)
	    \label{eq:sec2-fiber-integral}
	\end{equation}
    is finite, strictly positive, and continuous on $G(n,k)$.
    Consequently, $V\mapsto\log J(V)$ is continuous.
\end{lemma}

\begin{proof}
    Let $V_j\to V$, and let $O_j:V\to V_j$ be the isometries from Lemma~\ref{lem:sec2-nearby-subspace-isometry}.
    The pushforward identity~\eqref{eq:sec2-isometry-pushforward-measure} gives
	\begin{equation}
	    \int_{V_j}F(V_j,z)\,d\lambda_{V_j}(z)
	    =
	    \int_VF(V_j,O_jy)\,d\lambda_V(y).
	    \label{eq:sec2-moving-fiber-pullback}
	\end{equation}
    For each $y\in V$, we have $(V_j,O_jy)\to(V,y)$ by~\eqref{eq:sec2-nearby-isometry-convergence}, and therefore $F(V_j,O_jy)\to F(V,y)$.
    Since $O_j$ is an isometry,
	\[
	    0<F(V_j,O_jy)\le\mathcal F(|O_jy|)=\mathcal F(|y|).
	\]
    The dominating function is integrable on $V$, because polar coordinates give
	\begin{equation}
	    \int_V\mathcal F(|y|)\,d\lambda_V(y)
	    =|S^{k-1}|\int_0^\infty\mathcal F(r)r^{k-1}\,dr<\infty.
	    \label{eq:sec2-radial-majorant-fiber-integral}
	\end{equation}
    Here $|S^{k-1}|$ denotes the surface area of $S^{k-1}$.
    The dominated convergence theorem is therefore applicable to the fixed-fiber integral in \eqref{eq:sec2-moving-fiber-pullback}, and it proves $J(V_j)\to J(V)$.

    Finiteness follows from the same dominating function.
    For positivity, continuity and $F(V,0)>0$ give a neighborhood of $0$ in $V$ on which $F(V,\cdot)$ has a positive lower bound.
    This neighborhood has positive $k$-dimensional Lebesgue measure, so $J(V)>0$.
    The last assertion follows by composition with the logarithm.
\end{proof}

\begin{lemma}
    \label{lem:sec2-measurability-continuous-integrand}
    Under the assumptions of Theorem~\ref{thm:sec1-forward-grassmannian-bl}, the map
	\begin{equation}
	    x\longmapsto
	    \exp\left(
	        \int_{G(n,k)}\log F(V,P_Vx)\,d\mu(V)
	    \right)
	    \label{eq:sec2-measurable-continuous-integrand}
	\end{equation}
    is Lebesgue measurable.
    The analogous map with $\mu$ replaced by any finite discrete measure is continuous.
\end{lemma}

\begin{proof}
    The map
	\[
	    (V,x)\longmapsto (V,P_Vx)
	\]
    is continuous from $G(n,k)\times\R^n$ to $T_{n,k}$.
    Hence
	\[
	    H(V,x):=\log F(V,P_Vx)
	\]
    is jointly continuous and therefore jointly Borel measurable.
    We denote
	\[
	    H^+(V,x)
	    :=
	    \max
	    \{
	        H(V,x),0
	    \}
	    \qquad
	    \text{and}
	    \qquad
	    H^-(V,x)
	    :=
	    \max
	    \{
	        -H(V,x),0
	    \}.
	\]
    Then $H^+$ and $H^-$ are nonnegative and measurable, so Tonelli's theorem shows that
	\[
	    x\longmapsto\int H^+(V,x)\,d\mu(V),
	    \qquad
	    x\longmapsto\int H^-(V,x)\,d\mu(V)
	\]
    are Lebesgue measurable.
    For every fixed $x$, the function $H(\cdot,x)$ is continuous on the compact Grassmannian and is therefore bounded.
    Since $\mu$ is finite by \eqref{eq:sec2-grassmannian-total-mass}, both integrals are finite.
    Their difference is consequently well defined, with no $\infty-\infty$ ambiguity, and exponentiation proves \eqref{eq:sec2-measurable-continuous-integrand}.
    For a finite discrete measure, the exponent is a finite sum of continuous functions of $x$.
\end{proof}

Lemmas~\ref{lem:sec2-projected-input-continuity}, \ref{lem:sec2-continuity-fiber-integral}, and \ref{lem:sec2-measurability-continuous-integrand} establish the three claims invoked in Subsection~\ref{subsec:sec2-formal-continuous-passage} and therefore complete the proof of Theorem~\ref{thm:sec1-forward-grassmannian-bl}.

\subsection{Compactly Supported Input Families}
\label{subsec:sec2-compact-input-families}

We now extend the Grassmannian inequality to compactly supported fiber inputs.
Throughout this subsection, $B_V(0,r)$ denotes the closed ball $\{y\in V:|y|\le r\}$.

\begin{proof}[Proof of Theorem~\ref{thm:sec1-compact-grassmannian-bl}]
Suppose that $(f_V)_{V\in G(n,k)}$ satisfies the hypotheses of Theorem~\ref{thm:sec1-compact-grassmannian-bl}.
We first define a family of cutoff functions.
For each $\delta>0$, define $\psi_\delta:[0,\infty)\times[0,\infty)\to[0,1]$ by
	\[
	    \psi_\delta(r,a)
	    =
	    \begin{cases}
	        1, & r\le a,\\
	        0, & r\ge a+\delta,\\
	        1-\dfrac{r-a}{\delta}, & a\le r\le a+\delta.
	    \end{cases}
	\]
We then define $\varphi_\delta:T_{n,k}\to[0,1]$ by
	\[
	    \varphi_\delta(V,y)
	    =
	    \psi_\delta\bigl(\alpha(V),|y|\bigr).
	\]
We also define a family of regularization functions.
For each $\eta>0$, define $\rho_\eta:T_{n,k}\to(0,\infty)$ by
	\[
	    \rho_\eta(V,y)
	    =
	    \eta e^{-|y|^2}.
	\]
For $\delta,\eta>0$, consider the input function
	\[
	    F_{\delta,\eta}(V,y)
	    :=
	    \varphi_\delta(V,y)F(V,y)
	    +
	    \rho_\eta(V,y).
	\]
For fixed $\delta,\eta>0$, the function $F_{\delta,\eta}$ is continuous and strictly positive on $T_{n,k}$.
For $\delta<1$ and all sufficiently small $\eta>0$, we have
	\[
	    0
	    <
	    F_{\delta,\eta}(V,y)
	    \le
	    \psi_1\bigl(\alpha_M,|y|\bigr)
	    \mathcal F(|y|)
	    +
	    e^{-|y|^2}.
	\]
Here,
	\[
	    \alpha_M
	    :=
	    \sup_{V\in G(n,k)}\alpha(V),
	\]
which is finite because $\alpha$ is continuous on the compact space $G(n,k)$.
The function on the right-hand side is an admissible dominating function.
Thus, $F_{\delta,\eta}$ satisfies all the hypotheses of Theorem~\ref{thm:sec1-forward-grassmannian-bl}.
Applying this theorem, we obtain, for each fixed $\delta,\eta>0$,
	\[
	    \int_{\R^n}
	        \exp\left(
	            \int_{G(n,k)}
	                \log F_{\delta,\eta}(V,P_Vx)
	            \,d\mu(V)
	        \right)
	    \,dx
	    \le
	    \exp\left(
	        \int_{G(n,k)}
	            \log\left(
	                \int_V
	                    F_{\delta,\eta}(V,y)
	                \,d\lambda_V(y)
	            \right)
	        \,d\mu(V)
	    \right).
	\]

We first let $\eta\to0$.
For every $(V,y)\in T_{n,k}$, we have
	\begin{align*}
	    \lim_{\eta\to0}
	        F_{\delta,\eta}(V,y)
	    &=
	    \lim_{\eta\to0}
	        \left(
	            \varphi_\delta(V,y)F(V,y)
	            +
	            \eta e^{-|y|^2}
	        \right)\\
	    &=
	    \varphi_\delta(V,y)F(V,y)
	    =:
	    F_{\delta,0}(V,y).
	\end{align*}
Thus, for each fixed $x\in\R^n$, we have the decreasing limit
	\[
	    \lim_{\eta\to0}
	        \exp\left(
	            \int_{G(n,k)}
	                \log F_{\delta,\eta}(V,P_Vx)
	            \,d\mu(V)
	        \right)
	    =
	    \exp\left(
	        \int_{G(n,k)}
	            \log F_{\delta,0}(V,P_Vx)
	        \,d\mu(V)
	    \right)
	    \ge0.
	\]
Here, the inner integral is allowed to equal $-\infty$.
The monotone convergence theorem therefore gives
	\begin{align*}
	    &\lim_{\eta\to0}
	        \int_{\R^n}
	            \exp\left(
	                \int_{G(n,k)}
	                    \log F_{\delta,\eta}(V,P_Vx)
	                \,d\mu(V)
	            \right)
	        \,dx\\
	    =\,&
	    \int_{\R^n}
	        \exp\left(
	            \int_{G(n,k)}
	                \log F_{\delta,0}(V,P_Vx)
	            \,d\mu(V)
	        \right)
	    \,dx.
	\end{align*}

For the right-hand side, we need a uniform $\mu$-integrable bound for
	\[
	    \left|
	        \log\left(
	            \int_V
	                F_{\delta,\eta}(V,y)
	            \,d\lambda_V(y)
	        \right)
	    \right|.
	\]
Since $\mu$ is finite, it is enough to find a finite uniform bound.
Since $\alpha$ is positive and continuous on a compact space, there exists $\alpha_m>0$ such that
	\[
	    \alpha(V)\ge\alpha_m
	\]
for every $V\in G(n,k)$.
By construction,
	\[
	    \int_V
	        F_{\delta,\eta}(V,y)
	    \,d\lambda_V(y)
	    \ge
	    \int_V
	        F_{\delta,0}(V,y)
	    \,d\lambda_V(y)
	    \ge
	    \int_{B_V(0,\alpha_m)}
	        F(V,y)
	    \,d\lambda_V(y).
	\]
The set
	\[
	    K
	    :=
	    \{(V,y)\in T_{n,k}:|y|\le\alpha_m\}
	\]
is a compact subset of $T_{n,k}$.
Since $F$ is continuous and strictly positive, there exists $F_m>0$ such that
	\[
	    F(V,y)\ge F_m
	\]
for every $(V,y)\in K$.
Thus,
	\[
	    \int_{B_V(0,\alpha_m)}
	        F(V,y)
	    \,d\lambda_V(y)
	    \ge
	    F_m\lambda_V\bigl(B_V(0,\alpha_m)\bigr)
	    =
	    F_m\lambda_{\R^k}\bigl(B_{\R^k}(0,\alpha_m)\bigr)
	    >
	    0.
	\]
Therefore, for every $\eta>0$, we have the uniform lower bound
	\begin{equation}
	    \log\left(
	        \int_V
	            F_{\delta,\eta}(V,y)
	        \,d\lambda_V(y)
	    \right)
	    \ge
	    \log\left(
	        F_m\lambda_V\bigl(B(0,\alpha_m)\bigr)
	    \right)
	    >
	    -\infty.
	    \label{eq:sec2-lb-lhs-cpt-input}
	\end{equation}
For the upper bound, for all sufficiently small $\delta,\eta>0$, we have
	\begin{align*}
	    \int_V
	        F_{\delta,\eta}(V,y)
	    \,d\lambda_V(y)
	    &\le
	    \int_V
	        \psi_1\bigl(\alpha_M,|y|\bigr)
	        \mathcal F(|y|)
	    \,d\lambda_V(y)
	    +
	    \int_V
	        e^{-|y|^2}
	    \,d\lambda_V(y)\\
	    &\le
	    \int_0^\infty
	        r^{k-1}
	        \left(
	            \psi_1(\alpha_M,r)\mathcal F(r)
	            +
	            e^{-r^2}
	        \right)
	    \,dr
	    <
	    \infty.
	\end{align*}
The last integral is finite and independent of $V$.
The dominated convergence theorem now gives
	\[
	    \lim_{\eta\to0}
	        \exp\left(
	            \int_{G(n,k)}
	                \log\left(
	                    \int_V
	                        F_{\delta,\eta}(V,y)
	                    \,d\lambda_V(y)
	                \right)
	            \,d\mu(V)
	        \right)
	    =
	    \exp\left(
	        \int_{G(n,k)}
	            \log\left(
	                \int_V
	                    F_{\delta,0}(V,y)
	                \,d\lambda_V(y)
	            \right)
	        \,d\mu(V)
	    \right).
	\]
It follows that
	\[
	    \int_{\R^n}
	        \exp\left(
	            \int_{G(n,k)}
	                \log F_{\delta,0}(V,P_Vx)
	            \,d\mu(V)
	        \right)
	    \,dx
	    \le
	    \exp\left(
	        \int_{G(n,k)}
	            \log\left(
	                \int_V
	                    F_{\delta,0}(V,y)
	                \,d\lambda_V(y)
	            \right)
	        \,d\mu(V)
	    \right).
	\]

Finally, we let $\delta\to0$.
As above, applying the monotone convergence theorem to the left-hand side gives
	\begin{align*}
	    &\lim_{\delta\to0}
	        \int_{\R^n}
	            \exp\left(
	                \int_{G(n,k)}
	                    \log F_{\delta,0}(V,P_Vx)
	                \,d\mu(V)
	            \right)
	        \,dx\\
	    =\,&
	    \int_{\R^n}
	        \exp\left(
	            \int_{G(n,k)}
	                \log\mathbf 1_{B(0,\alpha(V))}
	                F(V,P_Vx)
	            \,d\mu(V)
	        \right)
	    \,dx.
	\end{align*}
A similar argument for the right-hand side gives, for all sufficiently small $\delta>0$,
	\begin{align*}
	    \infty
	    >
	    \int_0^\infty
	        r^{k-1}
	        \psi_1(\alpha_M,r)
	        \mathcal F(r)
	    \,dr
	    \ge
	    \int_V
	        F_{\delta,0}(V,y)
	    \,d\lambda_V(y)
	    \ge
	    F_m
	    \lambda_{\R^k}\bigl(B_{\R^k}(0,\alpha_m)\bigr)
	    >
	    0.
	\end{align*}
The dominated convergence theorem therefore gives
	\begin{align*}
	    &\int_{\R^n}
	        \exp\left(
	            \int_{G(n,k)}
	                \log\mathbf 1_{B_V(0,\alpha(V))}
	                F(V,P_Vx)
	            \,d\mu(V)
	        \right)
	    \,dx\\
	    \le\,&
	    \exp\left(
	        \int_{G(n,k)}
	            \log\left(
	                \int_V
	                    \mathbf 1_{B_V(0,\alpha(V))}
	                    F(V,y)
	                \,d\lambda_V(y)
	            \right)
	        \,d\mu(V)
	    \right).
	\end{align*}
This proves Theorem~\ref{thm:sec1-compact-grassmannian-bl}.
\end{proof}

\section{A Heat-Semigroup Proof of the Continuous Grassmannian Brascamp--Lieb Inequality}
\label{sec:sec3-heat-semigroup-proof}

Our second proof is inspired by the heat-semigroup preservation argument of Barthe and Huet~\cite{BartheHuet2009}.
For a finite rank-one geometric datum, their argument starts from a comparison of the form $h(x)\le\prod_i f_i(\inner{x}{u_i})^{c_i}$.
A maximum-principle argument shows that the corresponding comparison $(e^{t\Delta}h)(x)\le\prod_i(e^{t\Delta}f_i)(\inner{x}{u_i})^{c_i}$ is preserved for every positive time.
The large-time asymptotics of the heat kernels then give the geometric Brascamp--Lieb inequality.
Heat-flow proofs of geometric Brascamp--Lieb inequalities also appear in the work of Carlen, Lieb, and Loss~\cite{CLL2004} and Bennett, Carbery, Christ, and Tao~\cite{BCCT}.

We extend the argument of Barthe and Huet from a finite geometric datum to an isotropic measure on $G(n,k)$.
The resulting statement is Theorem~\ref{thm:sec1-heat-semigroup-preservation}.
We first compute explicitly how the ambient heat semigroup acts on an input lifted from a fiber.
We write $P_{V^\perp}:=I_n-P_V$ for the orthogonal projection onto $V^\perp$.
For $V\in G(n,k)$ and a nonnegative Lebesgue measurable function $f_V$ on $V$, the orthogonal decomposition $\R^n=V\oplus V^\perp$, Pythagoras' theorem, and the Fubini--Tonelli theorem give
	\begin{align}
	    [e^{t\Delta}(f_V\circ P_V)](x)
	    &=(4\pi t)^{-n/2}
	    \int_{\R^n}
	        e^{-\frac{|x-z|^2}{4t}}f_V(P_Vz)
	    \,d\lambda_{\R^n}(z)\notag\\
	    &=(4\pi t)^{-n/2}
	    \int_V\int_{V^\perp}
	        e^{-\frac{|P_Vx-y|^2}{4t}}
	        e^{-\frac{|P_{V^\perp}x-w|^2}{4t}}
	        f_V(y)
	    \,d\lambda_{V^\perp}(w)\,d\lambda_V(y)\notag\\
	    &=(4\pi t)^{-k/2}
	    \int_V
	        e^{-\frac{|P_Vx-y|^2}{4t}}f_V(y)
	    \,d\lambda_V(y),
	    \label{eq:sec3-ambient-to-fiber-heat-flow}
	\end{align}
where the last equality follows from integrating the heat kernel
	\begin{equation}
	    \int_{V^\perp}
	        e^{-\frac{|P_{V^\perp}x-w|^2}{4t}}
	    \,d\lambda_{V^\perp}(w)
	    =
	    (4\pi t)^{(n-k)/2}.
	    \label{eq:sec3-orthogonal-gaussian-integral}
	\end{equation}
Thus the ambient heat flow of $f_V\circ P_V$ is constant along $V^\perp$ and equals the $k$-dimensional heat evolution of $f_V$ evaluated at $P_Vx$.
For consistency, every occurrence of $e^{t\Delta}$ below denotes the heat semigroup on $\R^n$.
In particular, when it is applied to $f_V\circ P_V$, we use \eqref{eq:sec3-ambient-to-fiber-heat-flow} instead of introducing separate notation for a heat semigroup on $V$.

In Subsection~\ref{subsec:sec3-preservation-to-inequality}, we show how Theorem~\ref{thm:sec1-heat-semigroup-preservation} implies the Grassmannian Brascamp--Lieb inequality.
We prove Theorem~\ref{thm:sec1-heat-semigroup-preservation} in Subsection~\ref{subsec:sec3-formal-proof}.
Finally, in Subsection~\ref{subsec:sec3-technical-details}, we justify the technical steps used in the proof, including measurability and the interchange of differentiation and integration, etc.

We must mention that Barthe and Huet also used this method to prove the reverse geometric Brascamp--Lieb inequality.
For a continuous isotropic measure on the Grassmannian, however, the corresponding argument would require us to extend the maximum-principle lemma from Euclidean spaces to an appropriate infinite-dimensional Banach space.
To the best of the author's knowledge, the lack of compactness in this setting means that such an extension would require additional assumptions.
For this reason, we do not discuss the reverse inequality in this paper.

\subsection{From preservation to the continuous inequality}
\label{subsec:sec3-preservation-to-inequality}

Assume for the moment that Theorem~\ref{thm:sec1-heat-semigroup-preservation} has been proved.
Let $(f_V)_{V\in G(n,k)}$ be either the strictly positive family $f_V(y)=F(V,y)$ from Theorem~\ref{thm:sec1-forward-grassmannian-bl} or the compactly supported family from Theorem~\ref{thm:sec1-compact-grassmannian-bl}.
Set
	\begin{equation}
	    g(x)
	    :=
	    \exp\left(
	        \int_{G(n,k)}
	            \log f_V(P_Vx)
	        \,d\mu(V)
	    \right).
	    \label{eq:sec3-target-function}
	\end{equation}
For $j\in\N$, define the bounded compactly supported function
	\begin{equation}
	    h_j(x)
	    :=
	    \min\{g(x),j\}\mathbf 1_{B_{\R^n}(0,j)}(x).
	    \label{eq:sec3-target-truncation}
	\end{equation}
Then $0\le h_j\le g$, and hence the initial comparison \eqref{eq:sec1-heat-semigroup-initial} holds with $h=h_j$.
Applying Theorem~\ref{thm:sec1-heat-semigroup-preservation} at $x=0$ and using \eqref{eq:sec3-ambient-to-fiber-heat-flow}, we obtain
	\begin{equation}
	   \label{eq:sec3-preservation-at-origin}
	    (4\pi t)^{-n/2}
	    \int_{\R^n}
	        e^{-\frac{|z|^2}{4t}}h_j(z)
	    \,dz
	    \le
	    \exp\left(
	        \int_{G(n,k)}
	            \log\left(
	                (4\pi t)^{-k/2}
	                \int_V
	                    e^{-\frac{|y|^2}{4t}}f_V(y)
	                \,d\lambda_V(y)
	            \right)
	        \,d\mu(V)
	    \right).
	\end{equation}
Taking traces in the isotropy identity \eqref{eq:sec1-grassmannian-isotropy} gives $\mu(G(n,k))=n/k$.
The factors $(4\pi t)^{-n/2}$ on the two sides of \eqref{eq:sec3-preservation-at-origin} therefore cancel, and we find
	\begin{equation}
	   \label{eq:sec3-cancelled-heat-normalization}
	    \int_{\R^n}
	        e^{-\frac{|z|^2}{4t}}h_j(z)
	    \,dz
	    \le
	    \exp\left(
	        \int_{G(n,k)}
	            \log\left(
	                \int_V
	                    e^{-\frac{|y|^2}{4t}}f_V(y)
	                \,d\lambda_V(y)
	            \right)
	        \,d\mu(V)
	    \right).
	\end{equation}

We next let $t\to\infty$.
For every $V\in G(n,k)$, by dominated convergence,
	\[
	    \int_V
	        e^{-\frac{|y|^2}{4t}}f_V(y)
	    \,d\lambda_V(y)
	    \longrightarrow
	    \int_Vf_V(y)\,d\lambda_V(y)
	\]
This convergence also passes through the logarithmic $\mu$-integral.
Indeed, the dominating function $\mathcal{F}$ gives a finite uniform upper bound, while the positivity of $F$ on a fixed ball in every fiber gives a strictly positive uniform lower bound for all $t\ge1$.
The details are in Lemma~\ref{lem:sec3-large-time-fiber-limit}.
Since $h_j$ is bounded and compactly supported, dominated convergence on $\R^n$ now gives
	\begin{equation}
	    \int_{\R^n}h_j(z)\,dz
	    \le
	    \exp\left(
	        \int_{G(n,k)}
	            \log\left(
	                \int_Vf_V(y)\,d\lambda_V(y)
	            \right)
	        \,d\mu(V)
	    \right).
	    \label{eq:sec3-truncated-grassmannian-bl}
	\end{equation}
The sequence $(h_j)$ increases pointwise to $g$.
The monotone convergence theorem therefore allows $j\to\infty$ in \eqref{eq:sec3-truncated-grassmannian-bl}.
For $f_V(\cdot)=F(V,\cdot)$, this proves Theorem~\ref{thm:sec1-forward-grassmannian-bl}.
For the compactly supported family, it proves Theorem~\ref{thm:sec1-compact-grassmannian-bl}.

\subsection{Formal proof of heat-semigroup preservation}
\label{subsec:sec3-formal-proof}

We now explain the proof of Theorem~\ref{thm:sec1-heat-semigroup-preservation}.
We use Lemma~\ref{lem:sec1-barthe-huet-zhang-maximum-principle}.
We first consider the compactly supported family $(f_V)_{V\in G(n,k)}$.
If $h=0$ almost everywhere, the conclusion is immediate.
Otherwise, $e^{t\Delta}h$ is strictly positive for every positive time $t>0$.
For $(t,x)\in(0,\infty)\times\R^n$, set
	\begin{equation}
	    H(t,x)
	    :=
	    \log[e^{t\Delta}h](x),
	    \qquad
	    F_V(t,x)
	    :=
	    \log[e^{t\Delta}(f_V\circ P_V)](x),
	    \label{eq:sec3-logarithmic-heat-flows}
	\end{equation}
and define
	\begin{equation}
	    C(t,x)
	    :=
	    \int_{G(n,k)}
	        F_V(t,x)
	    \,d\mu(V)
	    -
	    H(t,x).
	    \label{eq:sec3-comparison-function}
	\end{equation}
The desired preservation inequality~\eqref{eq:sec1-heat-semigroup-positive-time} is equivalent to $C(t,x)\ge0$ for all $x\in \R^n$ and $t > 0$.

Both $e^H$ and $e^{F_V}$ solve the heat equation on $\R^n$.
Consequently, we have the following evolution identities
	\begin{equation}
	    \partial_tH
	    =
	    \Delta_{\R^n}H+|\nabla_{\R^n}H|^2,
	    \qquad
	    \partial_tF_V
	    =
	    \Delta_{\R^n}F_V+|\nabla_{\R^n}F_V|^2.
	    \label{eq:sec3-logarithmic-heat-evolution}
	\end{equation}
We differentiate $C$ against $t$ and apply~\eqref{eq:sec3-logarithmic-heat-evolution}, then we obtain
	\begin{equation}
	   \label{eq:sec3-comparison-evolution}
	    \partial_tC
	    =
	    \left(
	        \int_{G(n,k)}
	            \Delta_{\R^n}F_V
	        \,d\mu(V)
	        -
	        \Delta_{\R^n}H
	    \right)
	    +
	    \left(
	        \int_{G(n,k)}
	            |\nabla_{\R^n}F_V|^2
	        \,d\mu(V)
	        -
	        |\nabla_{\R^n}H|^2
	    \right).
	\end{equation}

Suppose that $C$ satisfies the critical-point conditions in \eqref{eq:sec1-minimum-point-conditions} at $(t,x)$.
In particular, $\nabla_{\R^n}C(t,x)=0$, and hence
	\begin{equation}
	    \nabla_{\R^n}H(t,x)
	    =
	    \int_{G(n,k)}
	        \nabla_{\R^n}F_V(t,x)
	    \,d\mu(V).
	    \label{eq:sec3-critical-gradient-identity}
	\end{equation}
The condition $\Hess_{\R^n}C(t,x)\succeq0$ and the interchange of the Hessian with the $\mu$-integral imply
	\begin{equation}
	    \int_{G(n,k)}
	        \Delta_{\R^n}F_V(t,x)
	    \,d\mu(V)
	    -
	    \Delta_{\R^n}H(t,x)
	    =
	    \Delta_{\R^n}C(t,x)
	    \ge0.
	    \label{eq:sec3-critical-laplacian-inequality}
	\end{equation}

Since $[e^{t\Delta}(f_V\circ P_V)](x)$ depends only on $P_Vx$, we have $\nabla_{\R^n}F_V(t,x)\in V$.
Applying Lemma~\ref{lem:sec3-isotropic-section-estimate} and then using \eqref{eq:sec3-critical-gradient-identity}, we obtain
	\begin{align}
	    |\nabla_{\R^n}H(t,x)|^2
	    &=
	    \left|
	        \int_{G(n,k)}
	            \nabla_{\R^n}F_V(t,x)
	        \,d\mu(V)
	    \right|^2\notag\\
	    &\le
	    \int_{G(n,k)}
	        |\nabla_{\R^n}F_V(t,x)|^2
	    \,d\mu(V).
	    \label{eq:sec3-critical-gradient-inequality}
	\end{align}
Both terms on the right-hand side of \eqref{eq:sec3-comparison-evolution} are therefore nonnegative, and hence $\partial_tC(t,x)\ge0$ at every point satisfying \eqref{eq:sec1-minimum-point-conditions}.

To obtain the boundary condition at spatial infinity, we regularize the fiber heat flows.
For each $\delta>0$, define
	\[
	    f_V^\delta
	    :=
	    f_V+\delta
        \quad
        \text{and}
        \quad
	    F_V^\delta(t,x)
	    :=
	    \log[e^{t\Delta}(f_V^\delta\circ P_V)](x).
	\]
We also define
	\[
	    C^\delta(t,x)
	    :=
	    \int_{G(n,k)}
	        F_V^\delta(t,x)
	    \,d\mu(V)
	    -
	    H(t,x).
	\]
The argument above also applies to $C^\delta$.
Therefore, $C^\delta$ satisfies the first hypothesis in Lemma~\ref{lem:sec1-barthe-huet-zhang-maximum-principle}.

We now verify the second hypothesis.
Since the heat semigroup preserves constant functions, we have
	\begin{align*}
	    F_V^\delta(t,x)
	    &=
	    \log[e^{t\Delta}(f_V^\delta\circ P_V)](x)\\
	    &=
	    \log\left(
	        [e^{t\Delta}(f_V\circ P_V)](x)
	        +
	        [e^{t\Delta}\delta](x)
	    \right)\\
	    &=
	    \log\left(
	        [e^{t\Delta}(f_V\circ P_V)](x)
	        +
	        \delta
	    \right).
	\end{align*}
Since
	\(
	    [e^{t\Delta}(f_V\circ P_V)](x)\ge0,
	\)
we obtain
	\[
	    F_V^\delta(t,x)
	    \ge
	    \log\delta
	    >
	    -\infty.
	\]
Moreover, the bound $f_V(y)\le\mathcal F(|y|)$ gives
	\begin{align*}
	    0
	    &\le
	    [e^{t\Delta}(f_V\circ P_V)](x)\\
	    &=
	    (4\pi t)^{-k/2}
	    \int_V
	        \exp\left(
	            -\frac{|P_Vx-y|^2}{4t}
	        \right)
	        f_V(y)
	    \,d\lambda_V(y)\\
	    &\le
	    \|\mathcal F\|_{L^\infty([0,\infty))}.
	\end{align*}
Thus,
	\(
	    -\infty
	    <
	    \log\delta
	    \le
	    F_V^\delta(t,x)
	    \le
	    \log\left(
	        \|\mathcal F\|_{L^\infty([0,\infty))}
	        +
	        \delta
	    \right)
	    <
	    \infty.
	\)
After integrating with respect to $\mu$, we obtain
	\begin{align*}
	    \mu(G(n,k))\log\delta
	    \le
	    \int_{G(n,k)}
	        F_V^\delta(t,x)
	    \,d\mu(V)
	    \le
	    \mu(G(n,k))
	    \log\left(
	        \|\mathcal F\|_{L^\infty([0,\infty))}
	        +
	        \delta
	    \right).
	\end{align*}
On the other hand, $h$ is nonnegative, bounded, and compactly supported.
It follows that
	\[
	    \lim_{|x|\to\infty}
	        \sup_{t\in(0,\infty)}
	            [e^{t\Delta}h](x)
	    =
	    0.
	\]
Therefore, for every $M>0$, there exists $R_M>0$ such that
	\[
	    [e^{t\Delta}h](x)
	    \le
	    e^{-M}
	\]
whenever $|x|\ge R_M$ and $t>0$.
Consequently,
	\[
	    -H(t,x)
	    =
	    -\log[e^{t\Delta}h](x)
	    \ge
	    -\log(e^{-M})
	    =
	    M
	\]
for every $t>0$ and every $|x|\ge R_M$.
Choose $M$ sufficiently large that
	\[
	    \mu(G(n,k))\log\delta+M
	    \ge
	    0.
	\]
Then, for every $t>0$ and every $|x|\ge R_M$, we have
	\begin{align*}
	    C^\delta(t,x)
	    =
	    \int_{G(n,k)}
	        F_V^\delta(t,x)
	    \,d\mu(V)
	    -
	    H(t,x)
	    \ge
	    0.
	\end{align*}
In particular,
	\[
	    \liminf_{|x|\to\infty}
	        \left(
	            \inf_{t>0}C^\delta(t,x)
	        \right)
	    \ge
	    0.
	\]

The third hypothesis for $C^\delta$ follows from assumption~\eqref{eq:sec1-heat-semigroup-initial} and the inequality $f_V^\delta\ge f_V$.
Lemma~\ref{lem:sec1-barthe-huet-zhang-maximum-principle} therefore gives
	\[
	    C^\delta(t,x)
	    \ge
	    0
	\]
for every $(t,x)\in(0,\infty)\times\R^n$.
Finally, letting $\delta\downarrow0$ gives $C(t,x)\ge0$.
This completes the formal proof for the case of compactly supported input.

Next we assume Theorem~\ref{thm:sec1-heat-semigroup-preservation} is true for the case of $\alpha$ being finite and continuous.
We prove the case of $\alpha\equiv +\infty$.
Suppose that $f_V(y)=F(V,y)$.
Since $h$ is compactly supported, we can choose $R>0$ such that $h(x)=0$ whenever $|x|>R$.
For each $V\in G(n,k)$, define
	\[
	    f_V^R(y)
	    :=
	    \mathbf 1_{B_V(0,R)}(y)f_V(y).
	\]
This family satisfies the compact-support assumptions with the constant radius function $\alpha(V)=R$.

We first check the initial comparison for the truncated inputs.
If $|x|\le R$, then $|P_Vx|\le|x|\le R$ for every $V\in G(n,k)$.
Hence,
	\[
	    f_V^R(P_Vx)=f_V(P_Vx).
	\]
The initial comparison therefore follows from~\eqref{eq:sec1-heat-semigroup-initial}.
If $|x|>R$, it holds because $h(x)=0$ and the right-hand side is nonnegative.
Thus, for every $x\in\R^n$,
	\[
	    h(x)
	    \le
	    \exp\left(
	        \int_{G(n,k)}
	            \log f_V^R(P_Vx)
	        \,d\mu(V)
	    \right).
	\]
We may therefore apply the compactly supported case of the theorem.

Since $0\le f_V^R\le f_V$ and the heat semigroup preserves order, we have
	\[
	    0
	    <
	    [e^{t\Delta}(f_V^R\circ P_V)](x)
	    \le
	    [e^{t\Delta}(f_V\circ P_V)](x)
	\]
for every $V\in G(n,k)$, $t>0$, and $x\in\R^n$.
Taking logarithms and integrating with respect to $\mu$, we obtain
	\begin{align*}
	    (e^{t\Delta}h)(x)
	    &\le
	    \exp\left(
	        \int_{G(n,k)}
	            \log\left(
	                [e^{t\Delta}(f_V^R\circ P_V)](x)
	            \right)
	        \,d\mu(V)
	    \right)\\
	    &\le
	    \exp\left(
	        \int_{G(n,k)}
	            \log\left(
	                [e^{t\Delta}(f_V\circ P_V)](x)
	            \right)
	        \,d\mu(V)
	    \right).
	\end{align*}
This proves~\eqref{eq:sec1-heat-semigroup-positive-time} for the strictly positive input family.

\subsection{Technical details}
\label{subsec:sec3-technical-details}

We now justify the analytic steps used in the formal proof.
First, for each $(t,x)\in(0,\infty)\times\R^n$, we prove that the map
	\[
	    V
	    \longmapsto
	    [e^{t\Delta}(f_V\circ P_V)](x)
	\]
is $\mu$-measurable.
This map appears on the right-hand side of~\eqref{eq:sec1-heat-semigroup-positive-time}.
We then prove that $C(t,x)$ is once differentiable in $t$ and twice differentiable in $x$.
The functions $H$ and $F_V$, for each fixed $V$, already belong to $C_t^1$ and $C_x^2$.
It therefore remains to justify interchanging $\partial_t$, $\nabla_x$, $\Delta_x$, and $\operatorname{Hess}_x$ with the integral
	\[
	    \int_{G(n,k)}F_V(t,x)\,d\mu(V).
	\]
We also prove~\eqref{eq:sec3-critical-gradient-inequality} in a separate lemma.
For later reference, set
	\begin{equation}
	    u_V(t,x)
	    :=
	    [e^{t\Delta}(F(V,\cdot)\circ P_V)](x).
	    \label{eq:sec3-fiber-heat-solution}
	\end{equation}
In this subsection, we only consider the compactly supported input case.

\begin{lemma}\label{lem:sec3-measurabilities of mu-integrands}
    The maps $V\mapsto [e^{t\Delta_{\R^n}}(f_V\circ P_V)] (x)$ is $\mu$-measurable for any $t > 0$ and $x\in \R^n$.
\end{lemma}
\begin{proof}
    Since $(V,z)\mapsto (V,P_Vz)$ is continuous from $G(n,k)\times \R^n$ to $T_{n,k}$, and since $(V,y)\mapsto f_V(y)$ is Borel on $T_{n,k}$,
    We then have
        \[
            (V,z)\mapsto f_V(P_Vz)
        \]
    is Borel.
    Therefore,
        \[
            (V,z)
            \mapsto
            (4\pi t)^{-n/2}
            \exp
                \left(
                    -\frac{|x-z|^2}{4t}
                \right)
            f_V(P_V z)
        \]
    is nonnegative and Borel.
    Then, by the Fubini--Tonelli theorem, the map
        \[
            V\mapsto
            \int_{\R^n}
                (4\pi t)^{-n/2}
                \exp
                    \left(
                        -\frac{|x-z|^2}{4t}
                    \right)
                f_V(P_V z)
            \,d\lambda_{\R^n}(z)
            =
            [e^{t\Delta_{\R^n}}(f_V\circ P_V)] (x)
        \]
    is $\mu$-measurable.
\end{proof}

The next lemmas show that we can differentiate the integral $\int_{G(n,k)} F_V(t,x)\,d\mu(V)$.

\begin{lemma}\label{lem:sec3-swipe time derivative and mu-integral}
    We fix $x\in \R^n$ and $t > 0$. For the function $F_V(t,x)$ defined above, we have
        \[
            \partial_t
                \int_{G(n,k)}
                    F_V(t,x)
                \,d\mu(V)
            =
            \int_{G(n,k)}
                \partial_t
                    F_V(t,x)
            \,d\mu(V).
        \]
\end{lemma}
\begin{proof}
    Recall the definition of differentiation:
        \[
            \partial_t F_V(t,x)
            =
            \lim_{q\to 0}
                \frac{F_V(t+q,x) - F_V(t,x)}{q}.
        \]
    Since $\partial_t F_V$ exists, we only need to show
        \[
            \lim_{q\to 0}
                \int_{G(n,k)}
                    \frac{F_V(t+q,x) - F_V(t,x)}{q}
                \,d\mu(V)
            =
            \int_{G(n,k)}
                \lim_{q\to 0}
                    \frac{F_V(t+q,x) - F_V(t,x)}{q}
            \,d\mu(V).
        \]
    By the mean value theorem, it is enough to find an $L^1(\mu)$ dominating function for $|\partial_t F_V(\tau,x)|$ on a small neighbourhood, i.e., to find $G\in L^1(\mu)$ such that
        \[
            \sup_{\tau\in [t/2,3t/2]}
                |\partial_t F_V(\tau,x)|
            \le 
            G(V).
        \]
    We let $\gamma^V_t(y) := (4\pi t)^{-k/2}e^{-|y|^2/4t}$ for any $y\in V$.
    By the chain rule and the computation for $u_V$ at the beginning of this section, we have
        \[
            \partial_t F_V(\tau,x)
            =
            \frac
                {\partial_t([\gamma^V_\tau * f_V](P_Vx))}
                {[\gamma^V_\tau * f_V](P_Vx)}.
        \]
    Our goal is to find a strictly positive lower bound for the denominator which depends only on the fixed $(t,x)$, and an upper bound for the numerator which also depends only on the fixed $(t,x)$.
    We compute
        \begin{align*}
            [\gamma^V_\tau * f_V](P_V x)
            =
            \int_V
                \gamma^V_\tau(P_Vx - y) f_V(y)
            \,d\lambda_V(y)
            =
            \int_{B_V(0,\alpha(V))}
                \gamma^V_\tau(P_Vx - y) f_V(y)
            \,d\lambda_V(y).
        \end{align*}
    Since $\alpha$ is continuous on $G(n,k)$, which is compact, we can find
        \[
            \alpha_m 
            := 
            \inf_{V\in G(n,k)}
                \alpha(V)
            >
            0.
        \]
    We have
        \[
            \int_{B_V(0,\alpha(V))}
                \gamma^V_\tau(P_Vx - y) f_V(y)
            \,d\lambda_V(y)
            \ge
            \int_{B_V(0,\alpha_m)}
                \gamma^V_\tau(P_Vx - y) f_V(y)
            \,d\lambda_V(y).
        \]
    Note that $f_V(y) = F(V,y)$ on $B_V(0,\alpha_m)$ for all $V$.
    Moreover, $F$ is a positive continuous function on the compact set
        \[
            \{(V,y)\in G(n,k)\times B_{\R^n}(0,\alpha_m): y\in V\}.
        \]
    Thus we can find
        \[
            F_m
            :=
            \inf_{(V,y)}
                F(V,y)
            >
            0.
        \]
    Therefore,
        \[
            \int_{B_V(0,\alpha_m)}
                \gamma^V_\tau(P_Vx - y) f_V(y)
            \,d\lambda_V(y)
            \ge
            F_m
            \int_{B_V(0,\alpha_m)}
                \gamma^V_\tau(P_Vx - y)
            \,d\lambda_V(y).
        \]
    For fixed $(\tau, x)$, the map
        \[
            (V,y)
            \mapsto
            \gamma^V_\tau(P_Vx - y)
        \]
    is also a positive continuous function on 
        \[
            \{(V,y)\in G(n,k)\times B_{\R^n}(0,\alpha_m): y\in V\}.
        \]
    Thus we can find a quantity
        \[
            m_1(\tau,x)
            :=
            \inf_{(V,y)}
                \gamma^V_\tau(P_Vx - y)
            >
            0.
        \]
    Therefore,
        \[
            F_m
            \int_{B_V(0,\alpha_m)}
                \gamma^V_\tau(P_Vx - y)
            \,d\lambda_V(y)
            \ge
            F_m
            m_1(\tau,x)
            \lambda_V(B_V(0,\alpha_m)).
        \]
    Since $\lambda_V$ is the Lebesgue measure induced by $\lambda_{\R^n}$ and $B_V(0,\alpha_m)$ is defined by the metric on $V$ induced by the standard Euclidean metric on $\R^n$, $\lambda_V(B_V(0,\alpha_m))$ is independent of $V$.
    Therefore, for all fixed $(\tau,x)$, we can find a strictly positive quantity
        \[
            \tilde{m}_1(\tau,x):= F_m m_1(\tau,x)\lambda_V(B_V(0,\alpha_m)),
        \]
    such that
        \[
            |\partial_t F_V(\tau,x)|
            \le
            \left|
            \frac
                {\partial_t([\gamma^V_\tau * f_V](P_Vx))}
                {\tilde{m}_1(\tau,x)}
            \right|.
        \]
    For the numerator, first note that 
        \[
            \partial_t([\gamma^V_\tau * f_V](P_Vx))
            =
            [(\partial_t\gamma^V_\tau) * f_V ](P_Vx).
        \]
    Then, by the triangle inequality, our assumption $f_V(y) \le \mathcal{F}(|y|)$, and the standard upper bound for $\int_V |\partial_t\gamma^V_\tau(P_Vx-y)|d\lambda_V(y)$, we have
        \[
            |\partial_t([\gamma^V_\tau * f_V](P_Vx))|
            \le
            \|\mathcal{F}\|_{L^\infty([0,\infty))}
            \frac{C_{k}}{\tau}, \quad\text{for some positive constant }C_k.
        \]
    Combining the estimates for the denominator and the numerator, we have
        \[
            |\partial_tF_V(\tau,x)|
            \le
            \frac{C\|\mathcal{F}\|_{L^\infty([0,\infty))}}{\tilde{m}_1(\tau,x)\tau}.
        \]
    Since $\tau\in [t/2, 3t/2]$ and the map
        \[
            \big(\tau, (V,y)\big)
            \mapsto
            \gamma^V_\tau(P_Vx-y)
        \]
    is continuous on the compact set
        \[
            [t/2,3t/2]\times \{(V,y)\in G(n,k)\times B_{\R^n}(0,\alpha_m): y\in V\},
        \]
    there exists $\tilde{\tilde{m}}_1(t,x) > 0$ such that
        \[
            \tilde{m}_1(\tau,x) \ge \tilde{\tilde{m}}_1(t,x) \quad\text{for }\tau \in [t/2,3t/2].
        \]
    Therefore, 
        \[
            \sup_{\tau\in [t/2,3t/2]}
                |\partial_t F_V(\tau,x)|
            \le
            \frac{2C\|\mathcal{F}\|_{L^\infty([0,\infty))}}{\tilde{\tilde{m}}_1(t,x)t}.
        \]
    The dominating function is just a constant function in $V$ for fixed $(t,x)$, hence, it is $\mu$-integrable. 
    The dominated convergence theorem then gives the result.
\end{proof}

\begin{lemma}\label{lem:sec3-swipe nabla with mu-integral}
    We fix $x\in \R^n$ and $t > 0$. For the function $F_V(t,x)$ defined above, we have
        \[
            \nabla_{\R^n}
                \int_{G(n,k)}
                    F_V(t,x)
                \,d\mu(V)
            =
            \int_{G(n,k)}
                \nabla_{\R^n}
                    F_V(t,x)
            \,d\mu(V).
        \]
\end{lemma}
\begin{proof}
    We fix $x_0\in \R^n$ and $t_0 > 0$.
    We claim that, if there exists $G\in L^1(\mu)$ such that
        \[
            \sup_{x\in B_{\R^n}(x_0,\rho)}
                |\nabla_{\R^n} F_V(t_0,x)|
            \le
            G(V),\quad\text{for some }\rho > 0,
        \]
    for $\mu$-a.e. $V\in G(n,k)$, then the lemma follows.
    Indeed, we first consider the spatial Frech\'et derivative of $F_V$ at $(t_0,x_0)$:
        \[
            D_{\R^n}F_V(t_0,x_0): \R^n\to \R.
        \]
    It satisfies for any direction $p\in \R^n$,
        \[
            D_{\R^n}F_V(t_0,x_0)[p]
            =
            \lim_{\varepsilon\to 0}
                \frac{F_V(t_0,x_0 + \varepsilon p) - F_V(t_0,x_0)}{\varepsilon}.
        \]
    The Frech\'et derivative can be uniquely identified with the gradient through the inner product $\inner{\cdot}{\cdot}_{\R^n}$:
        \[
            D_{\R^n}F_V(t_0,x_0)[p]
            =
            \inner{\nabla_{\R^n}F_V(t_0,x_0)}{p}_{\R^n}.
        \]
    Now define a function $\varphi:[0,1]\to \R$ by $\varphi(s) = F_V(t_0,x_0 + s \varepsilon p)$.
    Note that $\varphi(0) = F_V(t_0,x_0)$ and $\varphi(1) = F_V(t_0,x_0 + \varepsilon p)$.
    Also, by the chain rule, we have
        \[
            \frac{d\varphi}{ds}
            =
            \inner{\nabla_{\R^n}F_V(t_0,x_0 + s \varepsilon p)}{\varepsilon p}_{\R^n}.
        \]
    Applying the fundamental theorem of calculus to $\varphi$, we have
        \[
            \frac{F_V(t_0,x_0 + \varepsilon p) - F_V(t_0,x_0)}{\varepsilon}
            =
            \int_0^1
                \inner{\nabla_{\R^n}F_V(t_0,x_0 + s \varepsilon p)}{p}_{\R^n}
            \,ds.
        \]
    In particular, we have
        \[
            \left|
                \frac{F_V(t_0,x_0 + \varepsilon p) - F_V(t_0,x_0)}{\varepsilon}
            \right|
            \le
            \sup_{s\in [0,1]}
                |\nabla_{\R^n}F_V(t_0,x_0 + s \varepsilon p)| |p|_{\R^n}.
        \]
    Thus, if we can find some $\rho > 0$ such that on the compact ball $B_{\R^n}(x_0,\rho)$,
        \[
            \sup_{x\in B_{\R^n}(x_0,\rho)}
                |\nabla_{\R^n}F_V(t_0,x)|
            \le
            G(V),
        \]
    then we may choose $\varepsilon$ small such that $x_0+s\varepsilon p\in B_{\R^n}(x_0,\rho)$.
    Therefore, we have
        \[
            \left|
                \frac{F_V(t_0,x_0 + \varepsilon p) - F_V(t_0,x_0)}{\varepsilon}
            \right|
            \le
            G(V)|p|_{\R^n}.
        \]
    The dominated convergence theorem gives
        \begin{align*}
            \lim_{\varepsilon\to 0}
                \int_{G(n,k)}
                    \frac{F_V(t_0,x_0 + \varepsilon p) - F_V(t_0,x_0)}{\varepsilon}
                \,d\mu(V)
            =
            \biginner
                {\int_{G(n,k)}
                    \nabla_{\R^n} F_V(t_0,x_0)
                \,d\mu(V)}
                {p}_{\R^n}.
        \end{align*}
    By the linearity of integration, the left-hand side is just
        \begin{align*}
            D_{\R^n}
                \left(
                    \int_{G(n,k)}
                        F_V(t_0,\cdot)
                    \,d\mu(V)
                \right)(x_0)[p]
            =
            \biginner
                {\nabla_{\R^n}
                    \left(
                        \int_{G(n,k)}
                            F_V(t_0,\cdot)
                        \,d\mu(V)
                    \right)(x_0)
                }
                {p}_{\R^n}.
        \end{align*}
    This shows that, for any unit vector $p\in \R^n$, we have
        \[
            \biginner
                {\int_{G(n,k)}
                    \nabla_{\R^n} F_V(t_0,x_0)
                \,d\mu(V)}
                {p}_{\R^n}
            =
            \biginner
                {\nabla_{\R^n}
                    \left(
                        \int_{G(n,k)}
                            F_V(t_0,\cdot)
                        \,d\mu(V)
                    \right)(x_0)
                }
                {p}_{\R^n}.
        \]
    Therefore,
        \[
            \nabla_{\R^n}
                \left(
                    \int_{G(n,k)}
                        F_V(t_0,\cdot)
                    \,d\mu(V)
                \right)(x_0)
            =
            \int_{G(n,k)}
                \nabla_{\R^n} F_V(t_0,x_0)
            \,d\mu(V),
        \]
    as claimed.
    It remains to find such a dominating function $G$.
    Recall the function $u_V(t,x)$ computed at the beginning of this section. By the chain rule, we have
        \[
            \nabla_{\R^n}
                F_V(t_0,x)
            =
            \frac{\nabla_{\R^n}u_V(t_0,x)}{u_V(t_0,x)}.
        \]
    If we let $\gamma^n_t(z):= (4\pi t)^{-n/2}e^{-|z|^2/4t}$ for any $z\in \R^n$, then
        \[
            u_V(t_0,x)
            =
            \int_{\R^n}
                \gamma^n_{t_0}(x-z)
                f_V(P_Vz)
            \,d\lambda_{\R^n}(z).
        \]
    For the numerator, we note that
        \[
            \nabla_{\R^n}
            \left(
                \int_{\R^n}
                    \gamma^n_{t_0}(x-z)
                    f_V(P_Vz)
                \,d\lambda_{\R^n}(z)
            \right)
            =
            \int_{\R^n}
                    [\nabla_{\R^n}
                        \gamma^n_{t_0}](x-z)
                f_V(P_Vz)
            \,d\lambda_{\R^n}(z).
        \]
    By our assumptions on $f_V$, we have
        \begin{align*}
            |\nabla_{\R^n}u_V(t_0,x)|
            &\le
            \|\mathcal{F}\|_{L^\infty([0,\infty))}
            \|\nabla_{\R^n}\gamma^n_{t_0}\|_{L^1(\R^n)}\\
            &\le
            C_{n,t_0}
            \|\mathcal{F}\|_{L^\infty([0,\infty))},\quad\text{for some constant }C_{n,t_0}.
        \end{align*}
    Next we consider the denominator.
    By the computation at the beginning of this section, $u_V$ can also be written as an integral over $V$, namely
        \[
            u_V(t_0,x)
            =
            \int_V
                \gamma^V_{t_0}(P_Vx - y)
                f_V(y)
            \,d\lambda_V(y).
        \]
    As in the proof of the previous lemma, we have
        \[
            u_V(t_0,x)
            \ge
            F_m
            \int_{B_V(0,\alpha_m)}
                \gamma^V_{t_0}(P_Vx - y)
            \,d\lambda_V(y).
        \]
    With the product topology, the function $(x,(V,y))\mapsto \gamma^V_{t_0}(P_Vx - y) > 0$ is continuous on
        \[
            B_{\R^n}(x_0,\rho)
            \times
            \{
                (V,y)\in G(n,k)\times B_{\R^n}(0,\alpha_m): y\in V
            \}
            \quad
            \text{for any fixed }\rho > 0.
        \]
    Moreover, this set is compact.
    Thus there exists a constant $m_2(t_0,x_0,\rho) > 0$ such that, for any $x\in B_{\R^n}(x_0,\rho)$,
        \[
            u_V(t_0,x)
            \ge
            F_m\lambda_V(B_V(0,\alpha_m)) m_2(t_0,x_0,\rho) =: \tilde{m}_2(t_0,x_0,\rho) > 0.
        \]
    Therefore, for this fixed $\rho > 0$, we have
        \[
            \sup_{x\in B_{\R^n}(x_0,\rho)}
                |\nabla_{\R^n} F_V(t_0,x)|
            \le
            \frac{C_{n,t_0}\|\mathcal{F}\|_{\infty}}{\tilde{m}_2(t_0,x_0,\rho)}.
        \]
    The dominating function is just a constant function in $V$ for the fixed $(t_0,x_0)$, hence it is $\mu$-integrable.
    The lemma follows from our claim.
\end{proof}

\begin{lemma}\label{lem:sec3-swipe hessian with mu integral}
    We fix $x\in \R^n$ and $t > 0$. 
    For the function $F_V(t,x)$ defined above, we have
        \[
            \operatorname{Hess}_{\R^n}
                \int_{G(n,k)}
                    F_V(t,x)
                \,d\mu(V)
            =
            \int_{G(n,k)}
                \operatorname{Hess}_{\R^n}
                    F_V(t,x)
            \,d\mu(V).
        \]
\end{lemma}
\begin{proof}
    We apply the same proof as in Lemma~\ref{lem:sec3-swipe nabla with mu-integral} to the first order spatial derivative of $F_V$.
    We fix $x_0\in \R^n$ and $t_0 > 0$.
    We first recall the definition of the second-order Frech\'et derivative of $F_V$:
        \[
            D_{\R^n}^2 F_V(t_0,x_0) : \R^n \to \mathcal{L}(\R^n,\R)
        \]
    is a linear map such that for any $q\in \R^n$
        \[
            D_{\R^n}F_V(t_0,x_0 + q)
            =
            D_{\R^n}F_V(t_0,x_0)
            +
            D_{\R^n}^2F_V(t_0,x_0)[q]
            +
            o(|q|).
        \]
    We usually identify $D_{\R^n}^2F_V(t_0,x_0)$ with a bilinear form in $\mathcal{B}(\R^n\times \R^n,\R)$: $D_{\R^n}^2F_V(t_0,x_0)[p,q]$ for any $(p,q)\in \R^n\times \R^n$.
    If $F_V$ is in $C^2$, then
        \[
            D_{\R^n}^2F_V(t_0,x_0)[p,q]
            =
            \inner{\operatorname{Hess}_{\R^n}F_V(t_0,x_0)p}{q}_{\R^n},
        \]
    where $\operatorname{Hess}_{\R^n}F_V(t_0,x_0) = (\partial_i\partial_j F_V(t_0,x_0))_{i,j=1}^n$.
    In particular, we have
        \[
            D_{\R^n}^2F_V(t_0,x_0)[p,p]
            =
            p^T\operatorname{Hess}_{\R^n}F_V(t_0,x_0) p.
        \]

    Now we choose any $p,q\in \R^n$.
    We should reduce the number of notations, but for convenience, let us define one more notation:
        \[
            I(t,x)
            :=
            \int_{G(n,k)}
                F_V(t,x)
            \,d\mu(V).
        \]
    We claim that, if there exists $G\in L^1(\mu)$ such that
        \[
            \sup_{x\in B_{\R^n}(x_0,\rho)}
                \| D_{\R^n}^2 F_V(t_0,x)\|_{\operatorname{op}}
            \le
            G(V),
            \quad
            \text{for some }\rho > 0
        \]
    for $\mu$-a.e. $V\in G(n,k)$, then Lemma~\ref{lem:sec3-swipe hessian with mu integral} follows.
    Let us prove this claim.
    By Lemma~\ref{lem:sec3-swipe nabla with mu-integral}, in some neighbourhood of $x_0$, we have
        \[
            D_{\R^n}
                I(t_0,x_0)
            [q]
            =
            \int_{G(n,k)}
                D_{\R^n}
                    F_V(t_0,x)
                [q]
            \,d\mu(V).
        \]
    Thus,
        \begin{align*}
            \frac
            {
                D_{\R^n}
                    I(t_0,x_0+\varepsilon p)
                [q]
            -
                D_{\R^n}
                    I(t_0,x_0)
                [q]
            }
            {\varepsilon}
            =\,
            \int_{G(n,k)}
                \frac
                {
                    D_{\R^n}F_V(t_0,x_0+\varepsilon p)[q]
                    -
                    D_{\R^n}F_V(t_0,x_0)[q]
                }
                {
                    \varepsilon
                }
            \,d\mu(V).
        \end{align*}
    Similarly, we conside a function $\varphi: [0,1]\to \R$ defined by $\varphi(s) = D_{\R^n}F_V(t_0,x_0 + s\varepsilon p)[q]$.
    Particularly, $\varphi(0) = D_{\R^n}F_V(t_0,x_0)[q]$ and $\varphi(1) = D_{\R^n}F_V(t_0,x_0 + \varepsilon p)[q]$.
    By chain rule, we have
        \[
            \frac{d \varphi(s)}{ds}
            =
            D_{\R^n}^2 F_V(t_0,x_0 + s \varepsilon p)[\varepsilon p,q].
        \]
    Thus, by the fundamental theorem of calculus, we have
        \[
            \varphi(1) - \varphi(0)
            =
            \int_0^1
                D_{\R^n}^2
                    F_V(t_0,x_0+s\varepsilon p)[\varepsilon p,q]
            \,ds.
        \]
    Divide both sides by $\varepsilon$ and substitute, we have
        \[
            \frac
            {
                D_{\R^n} F_V(t_0,x_0+\varepsilon p)[q]
                -
                D_{\R^n}F_V(t_0,x_0)[q]
            }
            {
                \varepsilon
            }
            =
            \int_0^1
                D_{\R^n}^2
                    F_V(t_0,x_0+s\varepsilon p)[p,q]
            \,ds.
        \]
    In particular, we have
        \[
            \left|
            \frac
            {
                D_{\R^n} F_V(t_0,x_0+\varepsilon p)[q]
                -
                D_{\R^n}F_V(t_0,x_0)[q]
            }
            {
                \varepsilon
            }
            \right|
            \le
            \sup_{s\in [0,1]}
                \|D_{\R^n}^2F_V(t_0,x_0+s\varepsilon p)\|_{\operatorname{op}}
            |p||q|.
        \]
    Thus, if we choose $\varepsilon$ small such that for all $0\le s \le 1$, $x_0+s\varepsilon p\in B_{\R^n}(x_0,\rho)$, we obtain
        \[
            \left|
            \frac
            {
                D_{\R^n} F_V(t_0,x_0+\varepsilon p)[q]
                -
                D_{\R^n}F_V(t_0,x_0)[q]
            }
            {
                \varepsilon
            }
            \right|
            \le
            G(V)|p||q|.
        \]
    By the dominated convergence theorem, we have
        \begin{align*}
            D_{\R^n}^2 I(t_0,x_0)
            &=
            \lim_{\varepsilon\to 0}
                \frac
            {
                D_{\R^n} I(t_0,x_0+\varepsilon p)[q]
                -
                D_{\R^n}I(t_0,x_0)[q]
            }
            {
                \varepsilon
            }\\
            &=
            \int_{G(n,k)}
                \lim_{\varepsilon\to 0}
                    \frac
                        {
                            D_{\R^n} F_V(t_0,x_0+\varepsilon p)[q]
                            -
                            D_{\R^n}F_V(t_0,x_0)[q]
                        }
                        {
                            \varepsilon
                        }
            \,d\mu(V)\\
            &=
            \int_{G(n,k)}
                D_{\R^n}^2
                    F_V(t_0,x_0)
                [p,q]
            \,d\mu(V).
        \end{align*}
    Thus, we have
        \[
            \inner{\operatorname{Hess}_{\R^n}I(t_0,x_0)p}{q}_{\R^n}
            =
            \int_{G(n,k)}
                \inner{\operatorname{Hess}_{\R^n}F_V(t_0,x_0)p}{q}
            \,d\mu(V).
        \]
    However, this holds for all $p,q\in \R^n$.
    Therefore,
        \[
            \operatorname{Hess}_{\R^n}I(t_0,x_0)
            =
            \int_{G(n,k)}
                \operatorname{Hess}_{\R^n}F_V(t_0,x_0)
            \,d\mu(V),
        \]
    as we claimed.
    It is left to find such dominating function $G$.
    By the chain rule and the triangle inequality, we have
        \begin{align*}
            \|\operatorname{Hess}_{\R^n} F_V(t_0,x)\|_{\operatorname{op}}
            &=
            \left\|
            \frac{\operatorname{Hess}_{\R^n}u_V(t_0,x)}{u_V(t_0,x)}
            -
            \frac{\nabla_{\R^n}u_V(t_0,x)\otimes \nabla_{\R^n}u_V(t_0,x)}{u_V(t_0,x)^2}
            \right\|_{\operatorname{op}}\\
            &\le
            \frac{\|\operatorname{Hess}_{\R^n}u_V(t_0,x)\|_{\operatorname{op}}}{|u_V(t_0,x)|}
            +
            \frac{\|\nabla_{\R^n}u_V(t_0,x)\otimes \nabla_{\R^n}u_V(t_0,x)\|_{\operatorname{op}}}{u_V(t_0,x)^2}.
        \end{align*}
    We choose some fixed $\rho > 0$, then the proof the Lemma~\ref{lem:sec3-swipe nabla with mu-integral} gives
        \[
            u_V(t_0,x) \ge \tilde{m}_2(t_0,x_0,\rho) > 0,\quad\text{for all }x\in B_{\R^n}(x_0,\rho),
        \]
    and
        \[
            |\nabla_{\R^n}u_V(t_0,x)| \le C_{n,t_0} \|\mathcal{F}\|_{\infty} < \infty.
        \]
    On the other hand, we note that
        \begin{align*}
            \operatorname{Hess}_{\R^n} u_V(t_0,x)
            &=
            \int_{\R^n}
                (\operatorname{Hess}_{\R^n}\gamma^n_{t_0})(x-z)f_V(P_Vz)
            \,d\lambda_{\R^n}(z).
        \end{align*}
    Thus,
        \[
            \|\operatorname{Hess}_{\R^n} u_V(t_0,x)\|_{\operatorname{op}}
            \le
            \|\mathcal{F}\|_{L^\infty([0,\infty))}
            \|\operatorname{Hess}_{\R^n}\gamma^n_{t_0}\|_{L^1(\R^n)}.
        \]
    We have find an estimation for $\|\operatorname{Hess}_{\R^n}\gamma^n_{t_0}\|_{L^1(\R^n)}$.
    The $ij$-entry of this Hessian matrix is
        \[
            \partial_i\partial_j
                \gamma^n_{t_0}(z)
            =
            \left(
                \frac{z_iz_j}{4t_0^2}
                -
                \frac{\delta_{ij}}{2t_0}
            \right)
            \gamma^n_{t_0}(z).
        \]
    Therefore, we consider the operator norm
        \[
            \|\operatorname{Hess}_{\R^n}\gamma^n_{t_0}(z)\|_{\operatorname{op}}
            \le
            \left(
                \frac{\|z\otimes z\|_{\operatorname{op}}}{4t_0^2}
                -
                \frac{\|I_d\|_{\operatorname{op}}}{2t_0}
            \right)
            \gamma^n_{t_0}(z)
            =
            \left(
                \frac{|z|^2}{4t_0^2}
                -
                \frac{1}{2t_0}
            \right)
            \gamma^n_{t_0}(z).
        \]
    Thus,
        \begin{align*}
            \|\operatorname{Hess}_{\R^n}\gamma^n_{t_0}\|_{L^1(\R^n)}
            &=
            \int_{\R^n}
                \|\operatorname{Hess}_{\R^n}\gamma^n_{t_0}(z)\|_{\operatorname{op}}
            \,d\lambda_{n}(z)\\
            &\le
            \int_{\R^n}
                \left(
                    \frac{|z|^2}{4t_0^2}
                    -
                    \frac{1}{2t_0}
                \right)
                \gamma^n_{t_0}(z)
            \,d\lambda_{n}(z)\\
            &\le
            \frac{1}{4t_0^2}\cdot 2n t_0 + \frac{1}{2t_0} =: \tilde{C}_{n,t_0}.
        \end{align*}
    Therefore, combine with other estimations, we have
        \[
            \|\operatorname{Hess}_{\R^n} F_V(t_0,x)\|_{\operatorname{op}}
            \le
            \frac{\tilde{C}_{n,t_0}\|\mathcal{F}\|_{L^\infty([0,\infty))}}{\tilde{m}_2(t_0,x_0,\rho)}
            +
            \frac{C_{n,t_0}^2
            \|\mathcal{F}\|_{L^\infty([0,\infty))}^2}{\tilde{m}_2(t_0,x_0,\rho)^2}.
        \]
    The right-hand side is just a constant function of $V$, hence $\mu$-integrable.
    Our claim proves this lemma.
\end{proof}

\begin{lemma}\label{lem:sec3-swipe laplacian with mu integral}
    We fix $x\in \R^n$ and $t > 0$. 
    For the function $F_V(t,x)$ defined above, we have
        \[
            \Delta_{\R^n}
                \int_{G(n,k)}
                    F_V(t,x)
                \,d\mu(V)
            =
            \int_{G(n,k)}
                \Delta_{\R^n}
                    F_V(t,x)
            \,d\mu(V).
        \]
\end{lemma}
\begin{proof}
    Follows directly from Lemma~\ref{lem:sec3-swipe hessian with mu integral} by taking trace.
\end{proof}

The next estimate is used in proving the nonnegativity of the first order part in~\eqref{eq:sec3-comparison-evolution}.
\begin{lemma}\label{lem:sec3-isotropic-section-estimate}
    Consider the map $Y:G(n,k)\to \R^n$ defined by $Y(V) = y_V\in V$. Assume that $Y$ is $\mu$-measurable and
        \[
            \int_{G(n,k)}
                |y_V|^2
            \,d\mu(V)
            <
            \infty.
        \]
    Then we have
        \[
            \left|
                \int_{G(n,k)}
                    y_V
                \,d\mu(V)
            \right|^2
            \le
            \int_{G(n,k)}
                |y_V|^2
            \,d\mu(V).
        \]
\end{lemma}
\begin{proof}
For convenience, we write 
    \(
        w
        =
        \int_{G(n,k)}
            y_V
        \,d\mu(V).
    \)
Then we apply the Cauchy--Schwarz inequality and compute
    \begin{align*}
        |w|^2
        &=
        \biginner{w}{\int_{G(n,k)} y_V \,d\mu(V)}\\
        &=
        \int_{G(n,k)}
            \inner{w}{y_V}
        \,d\mu(V)\\
        &=
        \int_{G(n,k)}
            \inner{w}{P_V^*P_V y_V}
        \,d\mu(V)\\
        &=
        \int_{G(n,k)}
            \inner{P_V w}{P_V y_V}
        \,d\mu(V)\\
        &\le
        \left(
            \int_{G(n,k)}
                |P_V w|^2
            \,d\mu(V)
        \right)^{1/2}
        \left(
            \int_{G(n,k)}
                |P_V y_V|^2
            \,d\mu(V)
        \right)^{1/2}.
    \end{align*}
On the other hand, the isotropy of $\mu$ gives
    \begin{align*}
        |w|^2
        &=
        \biginner{w}{\int_{G(n,k)} P_V w \,d\mu(V)}\\
        &=
        \int_{G(n,k)}
            \inner{w}{P_V w}
        \,d\mu(V)\\
        &=
        \int_{G(n,k)}
            \inner{P_V w}{P_V w}
        \,d\mu(V)\\
        &=
        \int_{G(n,k)}
            |P_V w|^2
        \,d\mu(V).
    \end{align*}
Comparing the above two computations, we have
    \[
        |w|
        \le
        \left(
            \int_{G(n,k)}
                |P_V y_V|^2
            \,d\mu(V)
        \right)^{1/2}.
    \]
Squaring both sides and noting $P_V y_V = y_V$ gives
    \[
        \left|
            \int_{G(n,k)}
                y_V
            \,d\mu(V)
        \right|^2
        \le
        \int_{G(n,k)}
            |y_V|^2
        \,d\mu(V).
    \]
\end{proof}

\begin{lemma}
    \label{lem:sec3-large-time-fiber-limit}
    Let $(f_V)$ be either input family in Theorem~\ref{thm:sec1-heat-semigroup-preservation}.
    For every fixed $x\in\R^n$,
	\begin{align}
	    &\lim_{t\to\infty}
	    \exp\left(
	        \int_{G(n,k)}
	            \log\left(
	                \int_V
	                    e^{-\frac{|P_Vx-y|^2}{4t}}f_V(y)
	                \,d\lambda_V(y)
	            \right)
	        \,d\mu(V)
	    \right)\notag\\
	    =\,&
	    \exp\left(
	        \int_{G(n,k)}
	            \log\left(
	                \int_Vf_V(y)\,d\lambda_V(y)
	            \right)
	        \,d\mu(V)
	    \right).
	    \label{eq:sec3-large-time-fiber-limit}
	\end{align}
\end{lemma}

\begin{proof}
    Fix $V\in G(n,k)$.
    For every $y\in V$,
	\[
	    e^{-\frac{|P_Vx-y|^2}{4t}}f_V(y)
	    \longrightarrow
	    f_V(y)
	    \qquad\text{as }t\to\infty.
	\]
    These functions are Borel measurable and satisfy
	\[
	    0
	    \le
	    e^{-\frac{|P_Vx-y|^2}{4t}}f_V(y)
	    \le
	    \mathcal F(|y|).
	\]
    The right-hand side belongs to $L^1(V,\lambda_V)$ by our hypothesis, so the dominated convergence theorem gives pointwise convergence in $V$ of the fiber integrals.
    The same majorant gives the uniform upper bound
	\[
	    \int_V
	        e^{-\frac{|P_Vx-y|^2}{4t}}f_V(y)
	    \,d\lambda_V(y)
	    \le
	    |S^{k-1}|
	    \int_0^\infty
	        \mathcal F(r)r^{k-1}
	    \,dr.
	\]
    For the lower bound, choose $a>0$ such that $f_V(y)=F(V,y)$ whenever $|y|\le a$, uniformly in $V$.
    In the strictly positive case any fixed $a>0$ is admissible, while in the compactly supported case we may take $a=\min_V\alpha(V)>0$.
    Positivity and continuity on the compact set $\{(V,y)\in T_{n,k}:|y|\le a\}$ give a constant $m_a>0$ such that $F(V,y)\ge m_a$ there.
    Hence, for $t\ge1$,
	\begin{align*}
	    \int_V
	        e^{-\frac{|P_Vx-y|^2}{4t}}f_V(y)
	    \,d\lambda_V(y)
	    &\ge
	    e^{-\frac{(|x|+a)^2}{4}}
	    m_a\lambda_{\R^k}(B_{\R^k}(0,a))
	    >
	    0.
	\end{align*}
    Thus the logarithms are bounded uniformly in $V$ and $t\ge1$.
    Since $\mu$ is finite, the uniform constant bound is $\mu$-integrable, and dominated convergence proves convergence of the logarithmic integrals.
    Continuity of the exponential function then proves \eqref{eq:sec3-large-time-fiber-limit}.
\end{proof}

\section{A Mixed-Rank Continuous Brascamp--Lieb Inequality}
\label{sec:sec4-mixed-rank-bl}

The two methods also apply when the dimensions of the subspaces vary.
We replace a single Grassmannian by the finite disjoint union of all Grassmannians of positive-dimensional subspaces of $\R^n$.
We present the definitions, the inequality, and the two formal proofs.

\subsection{Definitions and the Main Result}
\label{subsec:sec4-mixed-rank-definitions}

Let
	\begin{equation}
    	\mathcal G_n
        :=
        \bigsqcup_{k=1}^n 
            G(n,k),
    	\qquad 
        V
        \in
        \mathcal G_n.
	\label{eq:sec4-mixed-grassmannian}
	\end{equation}
For $V\in\mathcal G_n$, let $P_V$ be the orthogonal projection onto $V$.
We equip $\mathcal G_n$ with the metric
	\begin{equation}
	   d(V,W)
       :=
       \|P_V-P_W\|_{\mathrm{op}}.
	\label{eq:sec4-mixed-grassmannian-metric}
	\end{equation}
If $\dim V\ne \dim W$, then $d(V,W)=1$.
Thus this metric induces the disjoint union topology, and each $G(n,k)$ is both open and closed in $\mathcal G_n$.
In particular, $\mathcal G_n$ is compact and $r$ is continuous.

The input functions are defined on
	\begin{equation}
    	\mathcal T_n
    	:=
        \bigsqcup_{k=1}^n 
            T_{n,k}
    	=
        \{(V,y)\in\mathcal G_n\times\R^n:y\in V\}.
	\label{eq:sec4-mixed-tautological-space}
	\end{equation}
We use the metric $d(V,W)+|y-z|$ between $(V,y)$ and $(W,z)$ in $\mathcal T_n$.
Continuity on $\mathcal T_n$ is therefore equivalent to continuity on each $T_{n,k}$.
We write $\lambda_V$ for the $\dim V$-dimensional Lebesgue measure on $V$ induced by its Euclidean structure.
We also set $B_V(0,a):=\{y\in V:|y|\le a\}$ for $a>0$ and $B_V(0,\infty):=V$.

A finite positive Borel measure $\mu$ on $\mathcal G_n$ is called isotropic if
	\begin{equation}
    	\int_{\mathcal G_n}
            P_V
        \,d\mu(V)
        =
        I_n.
	\label{eq:sec4-mixed-isotropy}
	\end{equation}
Write $\mu_k:=\mu|_{G(n,k)}$.
Taking traces gives
	\begin{equation}
    	\int_{\mathcal G_n}
            \dim V
        \,d\mu(V)
    	=
        \sum_{k=1}^n 
        k\mu_k(G(n,k))
    	=
        n.
	\label{eq:sec4-mixed-dimension-identity}
	\end{equation}
Consequently, $1\le\mu(\mathcal G_n)\le n$.
The measures $\mu_k$ need not be isotropic separately.

\begin{theorem}
	\label{thm:sec4-forward-mixed-rank-bl}
	Let $\mu$ be an isotropic measure on $\mathcal G_n$.
	Suppose that $F:\mathcal T_n\to(0,\infty)$ is continuous.
	For each $k\in\{1,\ldots,n\}$, assume that there is a nonnegative measurable function $\mathcal F_k$ such that
	\begin{equation}
    	\mathcal F_k
        \in 
        L^\infty([0,\infty)),
    	\qquad
    	\int_0^\infty
            \mathcal F_k(s)s^{k-1}
        \,ds
        <
        \infty,
    	\qquad
    	F(V,y)
        \le
        \mathcal F_k(|y|)
        \quad
        \text{for }(V,y)
        \in 
        T_{n,k}.
	\label{eq:sec4-mixed-radial-majorants}
	\end{equation}
	Let $\alpha$ be either a continuous function from $\mathcal G_n$ to $(0,\infty)$ or identically equal to $\infty$.
	Define
	\begin{equation}
	    f_V(y)
        :=
        \mathbf 1_{B_V(0,\alpha(V))}(y)F(V,y).
	\label{eq:sec4-mixed-input-family}
	\end{equation}
	Then
	\begin{equation}
    	\int_{\R^n}
    	    \exp
                \left(
                    \int_{\mathcal G_n}
                        \log 
                            f_V(P_Vx)
                    \,d\mu(V)
                \right)
        \,dx
    	\le
    	\exp
            \left(
    	        \int_{\mathcal G_n}
    	            \log
                        \left(
                            \int_V 
                                f_V(y)
                            \,d\lambda_V(y)
                        \right)
                \,d\mu(V)
    	   \right).
	\label{eq:sec4-forward-mixed-rank-bl}
	\end{equation}
	We use the conventions $\log 0=-\infty$ and $e^{-\infty}=0$.
	The inner integral on the left may equal $-\infty$, while the right-hand side is finite and strictly positive.
\end{theorem}

If $\mu$ is supported on a single $G(n,k)$, this gives the corresponding equal-rank inequality.
The constant $1$ is sharp in the strictly positive case.
Indeed, the choice $F(V,y)=e^{-\pi|y|^2}$ gives equality by \eqref{eq:sec4-mixed-isotropy}.

\subsection{Proof by Measure Approximation}
\label{subsec:sec4-mixed-measure-approximation}

We first construct discrete approximations that preserve isotropy and the mass on each Grassmannian.

\begin{theorem}
	\label{thm:sec4-mixed-isotropic-approximation}
	Let $\mu$ be an isotropic measure on $\mathcal G_n$.
	There are finite discrete isotropic measures $\nu_\varepsilon$ on $\mathcal G_n$ such that
	\begin{equation}
    	\nu_\varepsilon
        \rightharpoonup
        \mu,
    	\qquad
    	\nu_\varepsilon(G(n,k))
        =
        \mu(G(n,k)),
    	\qquad 
        1\le k\le n.
	\label{eq:sec4-mixed-approximation-properties}
	\end{equation}
\end{theorem}

\begin{proof}
	Fix $\varepsilon>0$.
	For each $k$, choose a finite Borel partition
	\begin{equation}
	    G(n,k)
        =
        \bigsqcup_i 
            U_{k,i}^\varepsilon,
	\label{eq:sec4-mixed-small-cell-partitions}
	\end{equation}
    such that the diameter $\operatorname{diam}_d U_{k,i}^\varepsilon\le2\varepsilon$.
	Let $\{e_1,\ldots,e_n\}$ be the standard basis of $\R^n$ and set
	\[
    	X
        :=
        \operatorname{span}
            \left(
            	\{1\}
                \cup
            	\{V\mapsto\langle P_Ve_a,e_b\rangle:1\le a\le b\le n\}
    	    \right).
	\]
	Apply Theorem~\ref{thm:sec2-tchakaloff-quadrature} to the restrictions of these functions on every cell of positive $\mu$-measure.
	This gives nodes $V_{k,i,\ell}^\varepsilon\in U_{k,i}^\varepsilon$ and nonnegative weights $c_{k,i,\ell}^\varepsilon$ such that
	\begin{align}
    	\sum_\ell 
            c_{k,i,\ell}^\varepsilon
    	&=
        \mu(U_{k,i}^\varepsilon),
    \label{eq:sec4-mixed-local-mass}\\
    	\sum_\ell 
            c_{k,i,\ell}^\varepsilon 
            P_{V_{k,i,\ell}^\varepsilon}
    	&=
        \int_{U_{k,i}^\varepsilon}
            P_V
        \,d\mu(V).
	\label{eq:sec4-mixed-local-projections}
	\end{align}
	The number of nodes on each cell is at most $1+n(n+1)/2$.
	Cells of zero measure are omitted.
	Define
	\begin{equation}
    	\nu_\varepsilon
    	:=
        \sum_{k,i,\ell}
            c_{k,i,\ell}^\varepsilon
            \delta_{V_{k,i,\ell}^\varepsilon}.
	\label{eq:sec4-mixed-discrete-measure}
	\end{equation}
	Summing \eqref{eq:sec4-mixed-local-projections} gives
	\begin{equation}
    	\int_{\mathcal G_n}
            P_V
        \,d\nu_\varepsilon(V)
        =
        I_n.
	\label{eq:sec4-mixed-discrete-isotropy}
	\end{equation}
	Summing \eqref{eq:sec4-mixed-local-mass} over the cells in $G(n,k)$ gives $\nu_\varepsilon(G(n,k)) = \mu (G(n,k))$.

	For $g\in C(\mathcal G_n)$, let $\omega_g$ denote its modulus of continuity with respect to $d$.
	The nodes lie in their respective cells, and the total weight on each cell equals its mass.
	Therefore,
	\begin{align}
    	\left|
            \int_{\mathcal G_n}
                g
            \,d\nu_\varepsilon
            -
            \int_{\mathcal G_n}
                g
            \,d\mu\right|
    	&\le
        \sum_{k,i}
            \mu(U_{k,i}^\varepsilon)
            \operatorname{osc}_{U_{k,i}^\varepsilon}
                g
        \notag\\
    	&\le
        \mu(\mathcal G_n)
        \omega_g(2\varepsilon)
    	\longrightarrow
        0.
	\label{eq:sec4-mixed-weak-error}
	\end{align}
	Here $\operatorname{osc}_U g:=\sup_U g-\inf_U g$.
	This proves weak convergence.
\end{proof}

\begin{proof}[Proof of Theorem~\ref{thm:sec4-forward-mixed-rank-bl} by measure approximation]
	First suppose that $\alpha\equiv\infty$, so that $f_V(y)=F(V,y)$.
	After discarding zero weights, write the approximating measure as $\nu_\varepsilon=\sum_j c_j^\varepsilon\delta_{V_j^\varepsilon}$.
	Its isotropy identity is
	\[
	    \sum_j 
            c_j^\varepsilon P_{V_j^\varepsilon}
        =
        I_n.
	\]
	The finite geometric Brascamp--Lieb inequality applies to subspaces of different dimensions; see \cite{BCCT}.
	It gives
	\begin{equation}
    	\int_{\R^n}
    	    \exp
                \left(
                    \int_{\mathcal G_n}
                        \log 
                            F(V,P_Vx)
                    \,d\nu_\varepsilon(V)
                \right)
        \,dx
    	\le
    	\exp
            \left(
                \int_{\mathcal G_n}
                    \log 
                        \Big( 
                            \int_V
                                F(V,y)
                            \,d\lambda_V(y) 
                        \Big)
                \,d\nu_\varepsilon(V)
            \right).
	\label{eq:sec4-mixed-discrete-bl}
	\end{equation}
	For fixed $x$, the function $V\mapsto\log F(V,P_Vx)$ is continuous on $\mathcal G_n$.
	The continuity of $F$ and the radial bounds also imply that $V\mapsto \int_VF(V,y)\,d\lambda_V(y)$ is continuous and strictly positive on each $G(n,k)$.
	Thus $\log M_F$ is continuous on $\mathcal G_n$.
	Weak convergence gives the limits of both inner integrals in \eqref{eq:sec4-mixed-discrete-bl}.
	Fatou's lemma on the left then proves \eqref{eq:sec4-forward-mixed-rank-bl} for strictly positive inputs.

	Now suppose that $\alpha:\mathcal G_n\to(0,\infty)$ is continuous.
	For $\delta>0$, define
	\begin{equation}
	    \psi_\delta(s,a)
        :=
        \begin{cases}
        1, & s\le a,\\
        1-(s-a)/\delta, & a<s<a+\delta,\\
        0, & s\ge a+\delta.
        \end{cases}
	\label{eq:sec4-mixed-outer-cutoff}
	\end{equation}
	For $\delta,\eta>0$, set
	\begin{equation}
    	F_{\delta,\eta}(V,y)
    	:=
        \psi_\delta(|y|,\alpha(V))F(V,y)
        +
        \eta e^{-|y|^2}.
	\label{eq:sec4-mixed-positive-approximation}
	\end{equation}
	These functions are continuous and strictly positive on $\mathcal T_n$.
	On $T_{n,k}$, they are dominated by $\mathcal F_k(|y|)+\eta e^{-|y|^2}$.
	The strictly positive case therefore applies to each $F_{\delta,\eta}$.
	Let $F_{\delta,0}(V,y):=\psi_\delta(|y|,\alpha(V))F(V,y)$.
	We have the decreasing limits
	\[
    	F_{\delta,\eta}(V,y)
        \downarrow 
        F_{\delta,0}(V,y)
        \quad\text{as }\eta\downarrow 0,
    	\qquad
    	F_{\delta,0}(V,y)
        \downarrow 
        f_V(y)
        \quad\text{as }\delta\downarrow 0.
	\]
	Moreover, there are constants $0<a\le b<\infty$ such that
	\[
    	a
        \le
        \int_V
            F_{\delta,\eta}(V,y)
        \,d\lambda_V(y)
        \le 
        b,
    	\qquad 
        V
        \in
        \mathcal G_n,
        \quad 0<\delta,\eta\le1.
	\]
	The upper bound follows from \eqref{eq:sec4-mixed-radial-majorants}.
	For the lower bound, use $\min_{\mathcal G_n}\alpha>0$ and the positive minimum of $F$ on the corresponding compact part of $\mathcal T_n$.
	Since there are only finitely many possible dimensions, both bounds are uniform in $V$.
	Pass first to the limit $\eta\downarrow0$ and then to the limit $\delta\downarrow0$.
	Fatou's lemma applies on the left, and the uniform bounds give convergence of the logarithmic fiber integrals on the right.
	This proves \eqref{eq:sec4-forward-mixed-rank-bl} for compactly supported inputs.
\end{proof}

\subsection{Proof by Heat-Semigroup Preservation}
\label{subsec:sec4-mixed-heat-semigroup}

Throughout this subsection, $e^{t\Delta}$ denotes the heat semigroup on $\R^n$.
For $V\in\mathcal G_n$, its action on a projected input is
	\begin{equation}
    	[e^{t\Delta}(f_V\circ P_V)](x)
    	=
        (4\pi t)^{-\dim V/2}
    	\int_V
            \exp
                \left(
                    -\frac{|P_Vx-y|^2}{4t}
                \right)
                f_V(y)
        \,d\lambda_V(y).
	\label{eq:sec4-mixed-fiber-heat-flow}
	\end{equation}
Thus the dimension enters through the heat-kernel normalization on each fiber.

\begin{theorem}
	\label{thm:sec4-mixed-heat-semigroup-preservation}
	Let $\mu$ and $(f_V)_{V\in\mathcal G_n}$ satisfy the assumptions of Theorem~\ref{thm:sec4-forward-mixed-rank-bl}.
	Let $h:\R^n\to[0,\infty)$ be bounded, Lebesgue measurable, and compactly supported.
	Assume that
	\begin{equation}
    	h(x)
        \le
        \exp
            \left(\int_{\mathcal G_n}\log f_V(P_Vx)\,d\mu(V)\right)
    	\qquad\text{for every }x\in\R^n.
	\label{eq:sec4-mixed-heat-initial}
	\end{equation}
	Then, for every $t>0$ and $x\in\R^n$,
	\begin{equation}
    	[e^{t\Delta}h](x)
    	\le
        \exp
        \left(
    	    \int_{\mathcal G_n}
    	        \log
                    \left(
                        [e^{t\Delta}(f_V\circ P_V)](x)
                    \right)
            \,d\mu(V)
    	\right).
	\label{eq:sec4-mixed-heat-positive-time}
	\end{equation}
\end{theorem}

\begin{lemma}
	\label{lem:sec4-mixed-isotropic-section-estimate}
	Let $\mu$ be isotropic on $\mathcal G_n$.
	For every measurable section $V\mapsto a_V\in V$ with $\int_{\mathcal G_n}|a_V|^2\,d\mu(V)<\infty$, we have
	\begin{equation}
    	\left|
            \int_{\mathcal G_n}
                a_V
            \,d\mu(V)
        \right|^2
    	\le
        \int_{\mathcal G_n}
            |a_V|^2
        \,d\mu(V).
	\label{eq:sec4-mixed-section-estimate}
	\end{equation}
\end{lemma}

\begin{proof}
	For $z\in\R^n$, use $a_V\in V$, Cauchy--Schwarz, and isotropy to obtain
	\begin{align*}
    	\left|
            \left\langle
                \int_{\mathcal G_n}
                    a_V
                \,d\mu(V)
                ,
                z
            \right\rangle
        \right|
    	&=
        \left|
            \int_{\mathcal G_n}
                \langle a_V,P_Vz\rangle
            \,d\mu(V)
        \right|\\
    	&\le
        \left(
            \int_{\mathcal G_n}
                |a_V|^2
            \,d\mu(V)
        \right)^{1/2}
    	\left(
            \int_{\mathcal G_n}
                |P_Vz|^2
            \,d\mu(V)
        \right)^{1/2}\\
    	&=
        \left(
            \int_{\mathcal G_n}
                |a_V|^2
            \,d\mu(V)
        \right)^{1/2}
        |z|.
	\end{align*}
	Taking the supremum over $|z|=1$ proves the claim.
\end{proof}

\begin{proof}[Proof of Theorem~\ref{thm:sec4-mixed-heat-semigroup-preservation}]
	We first consider the strictly positive family $f_V(y)=F(V,y)$.
	If $h=0$ almost everywhere, the conclusion is immediate.
	Otherwise, $e^{t\Delta}h$ is strictly positive for $t>0$.
	Set
	\begin{equation}
    	H(t,x)
        :=
        \log
            [e^{t\Delta}h](x),
    	\qquad
    	F_V(t,x)
        :=
        \log
            [e^{t\Delta}(f_V\circ P_V)](x),
	\label{eq:sec4-mixed-logarithmic-flows}
	\end{equation}
	and define
	\begin{equation}
	    C(t,x)
        :=
        \int_{\mathcal G_n}
            F_V(t,x)
        \,d\mu(V)
        -
        H(t,x).
	\label{eq:sec4-mixed-comparison-function}
	\end{equation}
	The desired inequality is equivalent to $C(t,x)\ge0$.
	In the following calculation, all spatial derivatives act on $\R^n$.
	Both $e^H$ and $e^{F_V}$ solve the ambient heat equation, so
	\begin{equation}
    	\partial_t H
        =
        \Delta H
        +
        |\nabla H|^2,
    	\qquad
    	\partial_t F_V
        =
        \Delta F_V
        +
        |\nabla F_V|^2.
	\label{eq:sec4-mixed-logarithmic-evolution}
	\end{equation}
	Consequently,
	\begin{equation}
    	\partial_tC
    	=
        \int_{\mathcal{ G }_n}
            \Delta F_V
        \,d\mu(V)
        -
        \Delta H
        +
        \int_{\mathcal G_n}
            |\nabla F_V|^2
        \,d\mu(V)
        -
        |\nabla H|^2.
	\label{eq:sec4-mixed-comparison-evolution}
	\end{equation}

	Suppose that, at a point $(t,x)$,
	\[
    	C(t,x)
        \le
        0,
    	\qquad 
        \nabla C(t,x)=0,
    	\qquad 
        \Hess C(t,x)\succeq0.
	\]
	Then
	\begin{equation}
    	\nabla H(t,x)
        =
        \int_{\mathcal G_n}
            \nabla F_V(t,x)
        \,d\mu(V),
    	\qquad
    	\Delta C(t,x)
        \ge
        0.
	\label{eq:sec4-mixed-minimum-identities}
	\end{equation}
	By \eqref{eq:sec4-mixed-fiber-heat-flow}, $F_V(t,x)$ depends only on $P_Vx$, and hence $\nabla F_V(t,x)\in V$.
	Lemma~\ref{lem:sec4-mixed-isotropic-section-estimate} gives
	\begin{equation}
    	|\nabla H(t,x)|^2
    	=
        \left|
            \int_{\mathcal G_n}
                \nabla F_V(t,x)
            \,d\mu(V)
        \right|^2
    	\le
        \int_{\mathcal G_n}
            |\nabla F_V(t,x)|^2
        \,d\mu(V).
	\label{eq:sec4-mixed-critical-gradient}
	\end{equation}
	Both terms on the right of \eqref{eq:sec4-mixed-comparison-evolution} are therefore nonnegative.
	Thus $\partial_tC(t,x)\ge0$ at every such point.

	To apply Lemma~\ref{lem:sec1-barthe-huet-zhang-maximum-principle}, introduce the positive regularization $f_V^\delta:=f_V+\delta$, where $\delta>0$.
	Since the heat semigroup preserves constants, the corresponding comparison function is
	\begin{equation}
    	C^\delta(t,x)
    	:=
        \int_{\mathcal G_n}
    	    \log
                \left(
                    [e^{t\Delta}(f_V\circ P_V)](x)+\delta
                \right)
        \,d\mu(V)
        -
        H(t,x).
	\label{eq:sec4-mixed-regularized-comparison}
	\end{equation}
	The preceding differential calculation also applies to $C^\delta$.
	Moreover,
	\[
	    C^\delta(t,x)
        \ge
        \mu(\mathcal G_n)\log\delta
        -
        H(t,x).
	\]
	The compact support and boundedness of $h$ give
	\[
	    \lim_{|x|\to\infty}\sup_{t>0}
            [e^{t\Delta}h](x)
        =0.
	\]
	It follows that $C^\delta$ satisfies the spatial boundary condition in the maximum principle.

	For the initial condition, consider the function
    	\[
    	    x
            \mapsto 
            \exp
                \left(
                    \int_{\mathcal G_n}
                        \log(f_V(P_Vx)+\delta)
                    \,d\mu(V)
                \right).
    	\]
    This function is continuous in $x$.
	This function is bounded and continuous, and \eqref{eq:sec4-mixed-heat-initial} gives
        \[
            h(x)
            \le
            \exp
                \left(
                    \int_{\mathcal G_n}
                        \log(f_V(P_Vx)+\delta)
                    \,d\mu(V)
                \right).
        \]
    Since heat flow preserves the inequality, we have
        \begin{align*}
            C^{\delta}(t,x)
            &\ge
                \int_{\mathcal G_n}
                    \log(
                        [e^{t\Delta}
                            f_V\circ P_V](x)
                        +
                        \delta
                        )
                \,d\mu(V)
            -
            \log
            \left(
            e^{t\Delta}
                \left[
                    \exp
                        \left(
                            \int_{\mathcal G_n}
                                \log
                                    (
                                        f_V\circ P_V(\cdot))
                                        +
                                        \delta
                                    )
                            \,d\mu(V)
                        \right)
                \right]
            (x)
            \right)
        \end{align*}
    The initial convergence of the second part gives
        \begin{align*}
            &\liminf_{t\to 0, x\to x_0}
            C^{\delta}(t,x)\\
            \ge\,&
            \liminf_{t\to 0, x\to x_0}
                \int_{\mathcal G_n}
                    \log(
                        [e^{t\Delta}
                            f_V\circ P_V](x)
                        +
                        \delta
                        )
                \,d\mu(V)
            -
                \left(
                    \int_{\mathcal G_n}
                        \log(f_V\circ P_V (x_0)+\delta)
                    \,d\mu(V)
                \right)
        \end{align*}
    The dominated convergence theorem applies to the first term, we obtain
            \begin{align*}
                &\liminf_{t\to 0, x\to x_0}
                C^{\delta}(t,x)\\
                \ge\,&
                \left(
                    \int_{\mathcal G_n}
                        \log(f_V\circ P_V (x_0)+\delta)
                    \,d\mu(V)
                \right)
                -
                \left(
                    \int_{\mathcal G_n}
                        \log(f_V\circ P_V (x_0)+\delta)
                    \,d\mu(V)
                \right)\\
                =\,&
                0.
        \end{align*}
	The maximum principle now gives $C^\delta(t,x)\ge0$ for every $t>0$ and $x\in\R^n$.
	Letting $\delta\downarrow0$ proves \eqref{eq:sec4-mixed-heat-positive-time} for strictly positive inputs.

	For compactly supported inputs, use the functions $F_{\delta,\eta}$ from \eqref{eq:sec4-mixed-positive-approximation}.
	They satisfy $F_{\delta,\eta}(V,y)\ge f_V(y)$, so the same function $h$ satisfies the required initial comparison.
	Apply the strictly positive case to these families.
	Let $\eta\downarrow0$ and then $\delta\downarrow0$ in their heat flows and logarithmic averages.
	The limit is \eqref{eq:sec4-mixed-heat-positive-time} for the original compactly supported family.
\end{proof}

\begin{proof}[Proof of Theorem~\ref{thm:sec4-forward-mixed-rank-bl} by heat-semigroup preservation]
	Choose
        \[
            h(x)
            =
            \exp
                \left(
                    \int_{\mathcal G_n}
                        \log f_V(P_Vx)
                    \,d\mu(V)
                \right),
            \qquad
            h_R(x)
            :=
            \mathbf 1_{B_{\R^n}(0,R)}(x)h(x).
        \]
	The input assumptions imply that $h$ is bounded and nonnegative.
	Thus $h_R$ is an admissible input in Theorem~\ref{thm:sec4-mixed-heat-semigroup-preservation}.
	Apply that theorem at $x=0$.
	Formula~\eqref{eq:sec4-mixed-fiber-heat-flow} gives
    	\[
    	[e^{t\Delta}(f_V\circ P_V)](0)=(4\pi t)^{-\dim V/2}\int_Ve^{-|y|^2/(4t)}f_V(y)\,d\lambda_V(y).
    	\]
	The identity \eqref{eq:sec4-mixed-dimension-identity} implies
    	\begin{equation}
        	\exp
                \left(
                    -\frac{\log(4\pi t)}{2}
                    \int_{\mathcal G_n}
                        \dim V
                        \,d\mu(V)
                \right)
        	=
            (4\pi t)^{-n/2}.
    	\label{eq:sec4-mixed-heat-normalization}
    	\end{equation}
	Multiplying the preservation inequality by $(4\pi t)^{n/2}$ therefore gives
    	\begin{equation}
        	\int_{\R^n}
                e^{-|y|^2/(4t)}h_R(y)
            \,dy
        	\le
            \exp
                \left(
                    \int_{\mathcal G_n}
                    \log 
                        \Big(
                        \int_V
                            e^{-|y|^2/(4t)}
                            f_V(y)
                        \,d\lambda_V(y)
                        \Big)
                    \,d\mu(V)
                \right).
    	\label{eq:sec4-mixed-normalized-preservation}
    	\end{equation}
	As $t\to\infty$, we have
    	\[
        	\int_V
                e^{-|y|^2/(4t)}f_V(y)
            \,d\lambda_V(y)
            \longrightarrow 
            \int_V
                f_V(y)
            \,d\lambda_V(y).
    	\]
	Passing to the limit in \eqref{eq:sec4-mixed-normalized-preservation} yields
    	\[
    	\int_{\R^n}
            h_R(x)
        \,dx
    	\le
        \exp
            \left(
                \int_{\mathcal G_n}
                    \log 
                        \Big(
                        \int_V
                            f_V(y)
                        \,d\lambda_V(y)
                        \Big)
                \,d\mu(V)
            \right).
    	\]
	Finally, $h_R\uparrow h$ as $R\to\infty$.
	Monotone convergence proves \eqref{eq:sec4-forward-mixed-rank-bl}.
\end{proof}

\section*{Acknowledgements}

The author was supported by JST SPRING, Japan Grant Number JPMJSP2108. The author would like to thank his advisor Neal Bez for weekly discussions throughout this project.
The author also thanks Hiroshi Tsuji for valuable discussions.

\section*{Declaration of Generative AI Use}

During the preparation of this manuscript, the author used generative AI as an assistive tool for checking and improving English grammar, improving LaTeX formatting, and identifying and correcting typographical errors.
\bibliographystyle{plainnat}
\bibliography{bib}

\end{document}